\documentclass [11pt,oneside, reqno] {amsart}
\usepackage[leqno]{amsmath}
\usepackage{amsthm}
\usepackage{dsfont}
\usepackage{amsfonts}
\usepackage{amssymb}
\usepackage{eucal}
\usepackage{upgreek}
\usepackage{latexsym,amscd,fullpage, tikz, cmll, mathrsfs, epic, mathtools} 
\usepackage{extpfeil}
\usepackage{color}
\usepackage{graphicx}
\usepackage [all]{xy}
\usepackage{hyperref}
\usepackage{stmaryrd}
\usepackage{comment}
\usepackage{soul}
\usepackage{manfnt}
\usepackage{enumerate}

\usetikzlibrary{fit}
\usetikzlibrary{arrows}
\usetikzlibrary{matrix}
\SelectTips{cm}{12}

  \newtheorem{theorem}{Theorem}[subsection]
  \newtheorem{proposition}[theorem]{Proposition}
  \newtheorem{lemma}[theorem]{Lemma}
  \newtheorem{corollary}[theorem]{Corollary}
   \newtheorem{definition}[theorem]{Definition}
  \newtheorem{thm2}{Theorem}[section]

  \newtheorem*{theorem*}{Theorem}

\theoremstyle{remark} 
  \newtheorem{remark}[theorem]{Remark}
  \newtheorem{example}[theorem]{Example}
  
  \newtheorem{rem2}[thm2]{Remark}
  
\theoremstyle{definition}

\newcommand {\Hom} {\operatorname {Hom}}

\renewcommand {\geq} {\geqslant}

\newcommand {\tensor}{\otimes}
\newcommand{\Spec}[1]{\mathrm{Spec}\ \!#1}
\newcommand{\ZZ}{\mathbb{Z}}

\newcommand{\betterhookarrow}{\!\xymatrix{{}\ar@{^{(}->}[r]&{}}\! }

\renewcommand{\labelitemi}{{\bf (\textdagger)}}

\def\cS{\mathcal{S}}

\def\Sc{\mathcal{S}}

\def\Xc{\mathfrak{X}}

\def\QQ{\mathbb{Q}}
\def\RR{\mathbb{R}}
\def\ZZ{\mathbb{Z}}

\def\<{\langle}
\def\>{\rangle}

\DeclareMathOperator\init{in}
\DeclareMathOperator\Spa{Spa}
\DeclareMathOperator\Spf{Spf}
\DeclareMathOperator\trop{Trop}
\DeclareMathOperator\val{val}

\def\1{{\bf 1}}
\def\lra{\longrightarrow}

\newcommand{\an}[1]{#1^{\mathrm{an}}}
\newcommand{\ad}[1]{#1^{\mathrm{ad}}}

\newcommand{\Ooad}[1]{\mathscr{O}_{\!\ad{#1}}^{\circ}}
\newcommand{\Trop}[2]{\mathrm{Trop}(#1,#2)}
\newcommand{\fTrop}[2]{\mathfrak{Trop}(#1,#2)}

\newextarrow{\xbigtoto}{{20}{20}{20}{20}}
   {\bigRelbar\bigRelbar{\bigtwoarrowsleft\rightarrow\rightarrow}}

\DeclareSymbolFont{largesym}{OML}{cmm}{m}{it}
\DeclareMathSymbol{\nstnsmall}{0}{largesym}{"22}

\makeatletter
\def\@tocline#1#2#3#4#5#6#7{\relax
  \ifnum #1>\c@tocdepth 
  \else
    \par \addpenalty\@secpenalty\addvspace{#2}%
    \begingroup \hyphenpenalty\@M
    \@ifempty{#4}{%
      \@tempdima\csname r@tocindent\number#1\endcsname\relax
    }{%
      \@tempdima#4\relax
    }%
    \parindent\z@ \leftskip#3\relax \advance\leftskip\@tempdima\relax
    \rightskip\@pnumwidth plus4em \parfillskip-\@pnumwidth
    #5\leavevmode\hskip-\@tempdima
      \ifcase #1
       \or\or \hskip 1em \or \hskip 2em \else \hskip 3em \fi%
      #6\nobreak\relax
    \dotfill\hbox to\@pnumwidth{\@tocpagenum{#7}}\par
    \nobreak
    \endgroup
  \fi}
\makeatother
\usepackage{hyperref}

\title{Adic tropicalizations and cofinality of Gubler models}
\author{Tyler Foster}

\author{Sam Payne}
\address{ Department of Mathematics,
University of Texas at Austin, 
Austin, TX 78712}
\email{\href{mailto:sampayne@utexas.edu}{sampayne@utexas.edu}}
\begin{document}

\begin{abstract}
\vskip -.25cm
We introduce adic tropicalizations for subschemes of toric varieties as limits of Gubler models associated to polyhedral covers of the ordinary tropicalization.  Our main result shows that Huber's adic analytification of a subscheme of a toric variety is naturally isomorphic to the inverse limit of its adic tropicalizations, in the category of locally topologically ringed spaces.   The key new technical idea underlying this theorem is cofinality of Gubler models, which we prove for projective schemes and also for more general compact analytic domains in closed subschemes of toric varieties.  In addition, we introduce a $G$-topology and structure sheaf on ordinary tropicalizations, and show that Berkovich analytifications are limits of ordinary tropicalizations in the category of topologically ringed topoi.
\end{abstract}

\maketitle



\section{Introduction}\label{section: introduction}

Let $X$ be a separated scheme of finite type over a complete and algebraically closed, nontrivially valued field $K$.  In previous work, we related the underlying topological space of the nonarchimedean analytification $\an{X}$, in the sense of Berkovich \cite{Berkovich90}, to limits of tropicalizations for inverse systems of toric embeddings \cite{analytification, limits}.  Here, we deepen those results by studying analogous limits in the category of locally topologically ringed spaces, establishing connections to adic analytifications in the sense of Huber, and accounting for the structure sheaf on $\an{X}$.  

Our starting point is the well-known fact that tropical geometry provides a systematic method for producing formal models of $X$ from a toric embedding $\iota \colon X \hookrightarrow Y_\Sigma$ plus combinatorial data.  Roughly speaking, a polyhedral complex $\Delta$ that covers the ordinary tropicalization $\Trop{X}{\iota}$ gives rise to a formal model $\mathfrak{X}_\Delta$, which we refer to as a \emph{Gubler model}.  Furthermore, any refinement $\Delta'$ of $\Delta$ gives rise to a natural map of Gubler models $\mathfrak{X}_{\Delta'} \rightarrow \mathfrak{X}_{\Delta}$.  See \S\ref{Section-adic trop} for details of these constructions.
We define the \emph{adic tropicalization} to be the locally ringed space
\[
\fTrop{X}{\iota} = \varprojlim_\Delta \mathfrak{X}_\Delta,
\]
where the limit is taken over all morphisms of Gubler models induced by polyhedral refinements.

\medskip

We focus on inverse systems of toric embeddings $\Sc$ of $X$ that satisfy the following condition.

\vskip .2cm

\renewcommand{\labelitemi}{$(\dagger)$}
\begin{itemize}
\item There is an affine open cover $X = U_1 \cup \cdots \cup U_r$ such that, for any finite sets of regular functions $R_1 \subset K[U_1], \ldots, R_r \subset K[U_r]$, there is an embedding $\iota \in \Sc$ such that $U_j$ is the preimage of a torus invariant affine open set $U_{\sigma_j}$ and each element of $R_j$ is the pullback of a character that is regular on $U_{\sigma_j}$, for each $1 \leq j \leq r$.
\end{itemize}

\vskip .2cm

\noindent This condition is stronger than the condition $(\star)$ appearing in \cite{limits}.  While $(\star)$ is sufficient to recover the underlying topological space of $\an{X}$ as a limit of tropicalizations, a stronger condition seems necessary for recovering finer information, such as structure sheaves.   We emphasize that the inverse system of all toric embeddings of $X$ does satisfy ($\dagger$), with respect to any affine open cover, provided that $X$ admits a single toric embedding \cite[Theorem~4.2]{limits}. 

Given an inverse system of toric embeddings $\mathcal{S}$ satisfying ($\dagger$), our first result expresses the adic analytification $\ad{X}$ as a limit of adic tropicalizations in the category of locally topologically ringed spaces. (See \S\ref{sec:prelim} for background on Huber's adic spaces, including the construction of $\ad{X}$.)

\begin{thm2} \label{thm:adiclimit}
Let $\mathcal S$ be an inverse system of toric embeddings that satisfies \emph{($\dagger$)}. Then the induced map
\[
(\ad{X}, \Ooad{X}) \lra \varprojlim_{\mathcal S} \big(\ \!\fTrop{X}{\iota},\ \!\mathscr{O}\ \!\big)
\]
is an isomorphism of locally topologically ringed spaces.
\end{thm2}

\noindent This theorem is one way of making precise the idea that tropical geometry produces many formal models of algebraic varieties.  
When $X$ is projective, we prove that every formal model is dominated by a Gubler model; see Theorem~\ref{theorem: 2}.  

Tropical geometry also gives natural methods for constructing formal models of analytic subdomains in $\an{X}$.   Indeed, if $\mathfrak{X}_\Delta$ is a Gubler model, then each subcomplex $\Delta' \subset \Delta$ naturally gives rise to a formal open subscheme $\mathfrak{X}_{\Delta'} \subset \mathfrak{X}_{\Delta}$.  Then $\mathfrak{X}_{\Delta'}$ is an admissible formal model of the analytic subdomain $\mathrm{Trop}^{-1}(|\Delta'|) \subset \an{X}$, which is compact if $\Delta'$ is finite.

\begin{thm2}\label{theorem: cofinality of Gubler models}
	Let $\Sc$ be an inverse system of toric embeddings of $X$ that satisfies $(\dagger)$, and let $V \subset \an{X}$ be a compact analytic domain.  Then, for any formal model $\mathfrak{V}$ of $V$, there is a Gubler model $\mathfrak{X}_\Delta$ associated to some $\iota \in \mathcal{S}$, and a finite subcomplex $\Delta' \subset \Delta$ such that $\mathfrak{X}_{\Delta'}$ is an admissible formal model of $V$ that dominates $\mathfrak{V}$.
\end{thm2}

\noindent This does not prove existence of Gubler models with desirable properties such as mild singularities; even if $X$ is smooth and projective and $\mathfrak{V}$ is semistable, one does not have control over the singularities of the Gubler models $\mathfrak{X}_{\Delta'}$ that dominate $\mathfrak{V}$.

\bigskip

Arguments similar to those used in the proof of Theorem~\ref{thm:adiclimit} also give a natural way of realizing the Berkovich analytification $\an{X}$ as a limit of tropicalizations in a way that accounts for the structure sheaf.  Recall that the structure sheaf on $\an{X}$ is defined with respect to the $G$-topology (a Grothendieck topology in which the ``admissible opens" are subsets, fiber products are intersections, and admissible covers are a distinguished class of set-theoretic covers by admissible subsets), given by analytic domains and admissible covers.

	In \S\ref{section: The limit theorem for structure sheaves on Berkovich spaces}, we define an analogous $G$-topology on the ordinary tropicalization $\Trop{X}{\iota}$, 
	along with a structure sheaf $\mathscr{O}^{\text{trop}}$ of topological $R$-algebras in this $G$-topology, such that the tropicalization map $\trop \colon \an{X} \lra \big(\Trop{X}{\iota}, K_{\!} \otimes_{R\!} \mathscr{O}^{\text{trop}} \big)$ is a morphism of topologically $G$-ringed spaces, as are the projections between tropicalizations associated to morphisms of toric embeddings.  

\begin{thm2}\label{thm:berklimit}
Let $\mathcal S$ be an inverse system of toric embeddings that satisfies \emph{($\dagger$)}. Then the natural map
\[
\an{X}\ \lra\ \varprojlim_{\mathcal{S}} \big(\ \!\mathrm{Trop}(X, \iota)_G,\ \!K_{\!} \otimes_{R\!}  \mathscr{O}^{\mathrm{trop}} \ \!\big)
\]
induces an equivalence of locally ringed topoi.
\end{thm2}

\begin{rem2}
	If the residue field of $K$ is $\mathbb{C}$, and if $\iota \colon X\hookrightarrow Y_{\Sigma}$ is a closed embedding into a toric variety such that each initial degeneration $\init_{v}(X)$ is smooth, then the disjoint union of the initial degenerations $\init_{v}(X)(\mathbb{C})$, for $v \in \Trop{X}{\iota}$, is naturally identified with an object in Parker's category of exploded manifolds \cite{Parker12}. Thus our construction of the adic tropicalization $\fTrop{X}{\iota}$ provides both a Zariski topology and a sheaf of topological rings on any exploded manifold arising in this way. In situations where some of the initial degenerations $\init_{v}(X)$ are not smooth, or the residue field is not $\mathbb{C}$, our construction provides an algebro-geometric analogue of Parker's exploded manifolds, allowing singularities and non-reduced structures.  When $\dim X = 1$, adic tropicalization is closely related to the metrized curve complexes of Amini and Baker \cite{AminiBaker15}.  For additional remarks on the relationship between adic tropicalization and exploded manifolds, and a thorough discussion of the relation between adic tropicalization and metrized curve complexes, see \cite{Foster16}.
	\end{rem2}

\begin{rem2}
See also \cite{FosterRanganathan16b, FosterRanganathan16} for relations between adic analytifications and limits of tropicalizations over higher rank valued fields.
\end{rem2}

\noindent \textbf{Acknowledgments.}  We thank M. Baker, F. Baldassarri, B. Conrad, N. Friedenberg, W. Gubler, J. Rabinoff, and D. Ranganathan for helpful conversations related to this project. TF conducted research on this paper while a visiting researcher at L'Institut des Hautes \'Etudes Scientifiques, at L'Institut Henri Poincar\'e, and at the Max Planck Institute for Mathematics.  He was partially supported by NSF RTG grant DMS-0943832 and by Le Laboratoire d'Excellence CARMIN.  SP conducted research on this paper while visiting MSRI.  He was partially supported by NSF grants DMS-2001502 and DMS-2053261 and a Simons Fellowship.

\section{Preliminaries}  \label{sec:prelim}

Throughout, we fix an algebraically closed\footnote{
The assumption that $K$ is algebraically closed is used in the following ways.  First, in the introduction, we mention that the system of all toric embeddings of any closed subscheme of a toric variety satisfies ($\dagger$), referring to \cite[Theorem~4.2]{limits}.  The statement of that theorem should have included the hypothesis that $K$ be algebraically closed; it is required in the proof in order to apply the embedding algorithm from \cite{Wlodarczyk93}.  Next, $K$ being algebraically closed guarantees that the value group $\Gamma$ is divisible, which simplifies the discussion of $(\Gamma,\Sigma)$-admissible polyhedra in \S\ref{sec:admissible}.  If $\Gamma$ is neither discrete nor divisible, then these should also be required to have vertices in $N_\Gamma$.  Finally, in the proofs of our main results, we use the existence of locally finite completions of $(\Gamma,\Sigma)$-admissible complexes, as constructed in  \cite{ColesFriedenberg23}.  This construction produces complexes with vertices in the divisible hull of $\Gamma$, and we do not know how to overcome the resulting difficulties when $\Gamma$ is neither discrete nor divisible.} nonarchimedean field $K$, which is complete with respect to a nontrivial valuation $$\val \colon K\longrightarrow\mathbb{R}\sqcup\{\infty\}.$$ 
Let $R$ be the ring of integers in $K$, with maximal ideal $\mathfrak{m}$ and residue field $k=R/\mathfrak{m}$. Fix a real number $0<\varepsilon<1$, and let
$|-| \colon K\lra\mathbb{R}_{\geq0}$
denote the induced norm, given by $$|a|=\varepsilon^{\mathrm{val}(a)}.$$

In this section, we briefly recall the basic notions that we need from the theory of $K$-analytic spaces, in the sense of Berkovich \cite{Berkovich90}, admissible formal models, in the sense of Raynaud \cite{Bosch14}, and adic spaces in the sense of Huber \cite{Huber94}.  For an expanded expository presentation, motivated by the results of this paper and including further references, see also \cite[\S2]{Foster16}.

\subsection{Berkovich spectra}

Let $A \cong K \langle T_1, \ldots, T_n \rangle / \mathfrak a$ be a strictly affinoid algebra.  The Berkovich spectrum $\mathscr{M}(A)$ is the set of continuous seminorms $|-| \colon  A \rightarrow \RR_{\geq 0}$ that extend the given norm on $K$, equipped with the subspace topology for the inclusion $\mathscr{M}(A) \subset A^{\RR_{\geq 0}}$.  The Berkovich spectrum $\mathscr{M}(A)$ also carries a natural $G$-topology, which refines the ordinary topology, and can be described as follows.

Given elements $f_1, \ldots, f_m, g$ of $A$ with no common zero, let
\[
B = A\big\langle X_{1},\dots,X_{m}\big\rangle\big/(gX_{1}-f_{1},\ \!\dots,\ \!gX_{m}-f_{m})
\]
be the corresponding \emph{rational algebra}.  The presentation $A \big \langle X_1, \ldots, X_m \big \rangle \xtwoheadrightarrow{} B$ induces an inclusion $\mathscr{M}(B) \subset \mathscr{M}(A)$, and the subsets that occur in this way are called \emph{rational subdomains}.
A subset $V \subset \mathscr{M}(A)$ is \emph{admissible} in the $G$-topology if every point $x \in V$ has a neighborhood of the form $V_1 \cup \cdots \cup V_n$, where each $V_i$ is a rational domain, and $x \in V_1 \cap \cdots \cap V_n$.  Similarly, the \emph{admissible covers} in the $G$-topology are covers of admissible subsets by admissible subsets $V = \bigcup_{i \in I} V_{i}$ such that every point $x \in V$ has a neighborhood which is a finite union of sets in the cover $V_{i_1} \cup \cdots \cup V_{i_n}$, with $x \in V_{i_1} \cap \cdots \cap V_{i_n}.$  The structure sheaf on $\mathscr{M}(A)$ is the unique sheaf of topological rings in this $G$-topology whose value on a rational domain $\mathscr{M}(B)$ is the rational algebra $B$, and whose restriction maps between these rational algebras are the natural ones.  See \cite{Berkovich90, Berkovich93} for further details.


\subsection{Adic spectra} Given a strictly affinoid algebra $A \cong K \langle T_1, \ldots, T_n \rangle / \mathfrak a$, let $A^\circ \subset A$ be the subring of power bounded elements.  The {\em adic spectrum} of the pair $(A,A^{\circ})$, denoted $\Spa(A,A^{\circ})$, is the set of all equivalence classes of continuous seminorms $\|-\|_{x} \colon A\lra\{0\}\sqcup\Gamma$ such that $\|a\|_{x}\leq1$ for all $a\in A^{\circ}$.\footnote{The adic spectrum is defined similarly for pairs $(A, A^+)$, where $A^+$ is a subring of the power bounded elements \cite[\S1.1]{Huber96}, but we will only need the case where $A^+ = A^\circ$.}  Here $\Gamma$ denotes any totally ordered abelian group, written multiplicatively, and $\{0\}\sqcup\Gamma$ is the totally ordered abelian semigroup in which $0\cdot\gamma=0$ and $0<\gamma$ for all $\gamma\in\Gamma$, endowed with its order topology.

The equivalence relation is the smallest such that continuous seminorms
	$$
	\|-\|_{x} \colon A\lra\{0\}\sqcup\Gamma
	\ \ \ \ \ \ \ \mathrm{and}\ \ \ \ \ \ \ 
	\|-\|_{x'} \colon A\lra\{0\}\sqcup\Gamma',
	$$
are equivalent whenever there is an inclusion $\alpha \colon \{0\}\sqcup\Gamma \hookrightarrow \{0\}\sqcup\Gamma'$ of ordered abelian semigroups such that $\alpha{}_{{}^{\ \!\circ}}\|-\|_{x}=\|-\|_{x'}$.
	
If $B = A\big\langle X_{1},\dots,X_{m}\big\rangle\big/(gX_{1}-f_{1},\ \!\dots,\ \!gX_{m}-f_{m})$ is a rational algebra, with the topology induced by the quotient norm, then the presentation $A\langle X_1, \ldots X_m \rangle \xtwoheadrightarrow{} B$ induces an inclusion $\Spa(B, B^\circ) \subset \Spa(A, A^\circ)$.  The subsets that occur in this way are called \emph{rational subsets}, and rational subsets generate the topology on $\Spa(A, A^\circ)$.  

The structure sheaf $\mathscr{O}$ is the unique sheaf that takes the value $B$ on a rational subset $\Spa(B, B^\circ)$, and whose restriction maps between such rational algebras are the natural ones. The stalk of this structure sheaf $\mathscr O$ at any point $x \in \Spa(A, A^\circ)$ is a topological local ring, and the seminorm $\|-\|_{x}$ induces a continuous seminorm on $\mathscr{O}_{x}$.   We also consider the subsheaf $\mathscr{O}^\circ$ of power bounded functions, whose value on an open subset $U$ is the ring of sections $f\in\mathscr{O}_{X}(U)$ for which $\|f\|_{x}\leq1$ at every point $x\in U$.  If $U = \Spa(B, B^\circ)$ is a rational subset, then $\mathscr{O}^\circ(U) = B^\circ$.

\subsection{Analytic spaces}  \label{ss:Kanalytic}

A \emph{strictly $K$-analytic space} is a Hausdorff topological space $X$ with a net of compact subspaces $V_i \subset X$, each equipped with a homeomorphism to a strictly $K$-affinoid space, and, for each inclusion $V_j \subset V_i$, a morphism of $K$-affinoid spaces identifying $V_j$ with an affinoid domain in $V_i$. We then consider $X$ as a locally topologically $G$-ringed space, with the induced $G$-topology. 

An {\em adic space} over $K$ is a locally topologically ringed space $(Z,\mathscr{O}_{Z})$ with an atlas $\{V_{i}\betterhookarrow Z\}$, whose charts are adic spectra of strictly affinoid $K$-algebras. 

There is a natural functor from $K$-analytic spaces to adic spaces, defined as follows.  Suppose $V$ is a $K$-analytic space, with an atlas of Berkovich spectra $V_i \cong \mathscr{M}(A_i)$, glued along inclusions induced by morphisms $f_{ij }  \colon  A_{i} \rightarrow A_{j}$ for $V_j \subset V_i$.  Then the associated adic space $\ad{V}$ has an atlas of adic spectra $\Spa(A_i, A^\circ_i)$, glued along the inclusions induced by the $\{f_{ij} \}$.  If $X$ is a separated scheme of finite type over $K$, we write $\ad{X}$ for the \emph{adic analytification} of $X$, i.e., the adic space associated to the Berkovich analytification $\an{X}$.

\subsection{Admissible formal models}  Many of the relations between adic spaces and $K$-analytic spaces, including those arising through analytification and tropicalization of algebraic schemes, are best understood in terms of Raynaud's theory of admissible formal models.  For an accessible presentation of this theory, see \cite{Bosch14}.

A topological $R$-algebra $A$ is {\em admissible} if there is an isomorphism of topological rings
$$
A
\ \ \cong\ \ 
R\langle t_{1},\dots,t_{n}\rangle\big/\mathfrak{a},
$$
where $R\langle t_{1},\dots,t_{n}\rangle\big/\mathfrak{a}$ has its $\mathfrak{m}$-adic topology, such that:
\vskip .2cm
\begin{itemize}
\item[{\bf (i)}]
the ideal $\mathfrak{a}\subset R\langle t_{1},\dots,t_{n}\rangle$ is finitely generated;
\item[{\bf (ii)}]
\vskip .2cm
the ring $R\langle t_{1},\dots,t_{n}\rangle\big/\mathfrak{a}$ is free of $\mathfrak{m}$-torsion.
\end{itemize}
\vskip .2cm
An {\em admissible formal $R$-scheme} is a formal scheme over $\mathrm{Spf}_{\ \!\!}R$ that admits an open covering
by formal spectra $\mathfrak{U}_{i}=\mathrm{Spf}A_{i}$, such that each $A_{i}$ is an admissible $R$-algebra.  Throughout, we assume that all admissible formal schemes are separated and paracompact.

Let $\mathfrak{X}$ be an admissible formal $R$-scheme, and let $\big\{\mathfrak{U}_{i} \hookrightarrow\mathfrak{X}\big\}$ be an open cover of $\mathfrak{X}$ by formal spectra of admissible $R$-algebras $A_{i}$. Each topological algebra $K\tensor_{R}A_{i}$ is strictly $K$-affinoid \cite[\S7.4]{Bosch14}. Moreover, paracompactness of $\mathfrak{X}$ ensures that the Berkovich spectra $\mathscr{M}(K\tensor_{R}A_{i})$ of these strictly $K$-affinoid algebras glue to produce a $K$-analytic space over $K$ \cite[Proposition~1.3.3(b)]{Berkovich93}. We denote this $K$-analytic space $\an{\mathfrak{X}}$; it is the {\em Raynaud fiber} of $\mathfrak{X}$.

\begin{definition}\label{admissible model of Berkovich space}
{\bf (Admissible formal models of an analytic space).}
Let $V$ be any $K$-analytic space over $K$. An {\em admissible formal model of} $V$ is an admissible formal $R$-scheme $\mathfrak{V}$ together with an isomorphism of $K$-analytic spaces $\an{\mathfrak{V}}\xrightarrow{\ \sim\ }V$.
\end{definition}

\noindent A {\em morphism} of admissible formal models of $V$ is a morphism $\mathfrak{V}_{1}\lra\mathfrak{V}_{2}$ whose induced morphism $\mathfrak{V}_{1}^{\mathrm{an}}\lra\mathfrak{V}_{2}^{\mathrm{an}}$ of $K$-analytic spaces commutes with the isomorphisms $\mathfrak{V}_{1}^{\mathrm{an}}\xrightarrow{\ \sim\ }V$ and $\mathfrak{V}_{2}^{\mathrm{an}}\xrightarrow{\ \sim\ }V$.

\subsection{Quasi-compact adic spaces are limits of formal models}\label{specializations of adic spaces}

Let $\mathfrak{X}$ be an admissible formal $R$-scheme.  We write $\ad{\mathfrak{X}}$ for the adic space associated to $\an{\mathfrak{X}}$, and refer to this as the \emph{adic Raynaud fiber} of $\mathfrak{X}$.  It comes with a {\em specialization morphism}
\begin{equation*}
\mathrm{sp}_{\mathfrak{X}} \colon (\ad{\mathfrak{X}},\mathscr{O}^{\ \!\circ}_{\!\ad{\mathfrak{X}}})\lra(\mathfrak{X},\mathscr{O}_{\mathfrak{X}})
\end{equation*}
of locally topologically ringed spaces over $\mathrm{Spec}_{\ \!}K$.  Here, $\mathscr{O}^{\ \!\circ}_{\!\ad{\mathfrak{X}}} \subset \mathscr{O}_{\!\ad{\mathfrak{X}}}$ denotes the subsheaf of power bounded analytic functions on $\ad{\mathfrak{X}}$. 

When $\ad{\mathfrak{X}}$ is quasi-compact, these specialization morphisms induce a natural isomorphism
\[
(\ad{\mathfrak{X}}, \Ooad{X}) \xrightarrow{\sim} \varprojlim \mathfrak{X}'
\]
in the category of locally topologically ringed spaces, where the inverse limit is taken over the category of formal models $\mathfrak{X}'$ of $\an{\mathfrak{X}}$ \cite[Theorem~2.22]{Scholze12}.

\section{Adic tropicalization}\label{Section-adic trop}

We now define the adic tropicalization of a subscheme of a toric variety.  Its underlying set is the disjoint union of all initial degenerations.  This is equipped with the structure of a locally topologically ringed space, via an  identification with the inverse limit of all formal models associated to admissible polyhedral covers of the ordinary extended tropicalization.

\subsection{Tropicalization}\label{subsection: tropicalization}

We briefly recall the basic notion of the ordinary \emph{extended tropicalization}, hereafter referred to as \emph{tropicalization}, for subschemes of toric varieties.  See \cite[\S3]{analytification} for further details.  Let $Y_\Sigma$ be the toric variety over $K$ associated to a fan $\Sigma$ in $N_\RR$.  Each cone $\sigma \in \Sigma$ corresponds to an affine torus-invariant open subvariety $U_\sigma \subset Y_\Sigma$, whose coordinate ring is the semigroup ring $K[S_\sigma]$ generated over $K$ by the semigroup $S_\sigma$ of characters of the dense torus that extend to regular functions on $U_\sigma$.  We equip $\RR \sqcup \{\infty\}$ with the topology that makes the exponential map $\RR \sqcup \{\infty\}\lra\RR_{\ge0}$, given by $x \mapsto e^{-x}$, a homeomorphism. The tropicalization of $U_\sigma$ is the space of semigroup homomorphisms
\[
\trop (U_\sigma) = \Hom(S_\sigma, \RR \sqcup \{\infty\}),
\]
with the topology induced by that of $\RR \sqcup \{\infty\}$.

Just as $U_\sigma$ decomposes as a disjoint union of torus orbits $O_\tau$ corresponding to the faces $\tau \preceq \sigma$, the tropicalization $\text{Trop}(U_{\sigma})$ decomposes as a disjoint union of real vector spaces
\[
\text{Trop}(U_{\sigma}) = \bigsqcup_{\tau \preceq \sigma} N_\RR / \mathrm{span}(\tau).
\]
Here each $N_\RR/\text{span}(\tau)$ is canonically identified with the tropicalization $\trop(O_\tau)$ of the open torus orbit $O_\tau\subset U_\sigma$ corresponding to $\tau$. Each inclusion of faces $\tau \preceq \sigma$ induces open immersions $U_\tau \subset U_\sigma$ and $\trop(U_\tau) \subset \text{Trop}(U_{\sigma})$. The tropicalization $\trop(Y_\Sigma)$ is obtained by gluing $\{ \text{Trop}(U_{\sigma}) \}_{\sigma \in \Sigma}$ along the open immersions $\trop(U_\tau) \subset \text{Trop}(U_{\sigma})$, just as the toric variety $Y_\Sigma$ is obtained by gluing $\{ U_\sigma \}_{\sigma \in \Sigma}$ along the open immersions $U_\tau \subset U_\sigma$. The natural tropicalization maps on affine open toric subvarieties glue to give a proper continuous surjection
	\begin{equation*}
	\trop \colon  \an{Y_\Sigma} \xtwoheadrightarrow{} \trop(Y_\Sigma).
	\end{equation*}

Now, and for the remainder of the paper, we fix a separated scheme $X$ of finite type over $K$.  Let $\iota \colon  X \hookrightarrow Y_\Sigma$ be a closed embedding in a toric variety.  Then the tropicalization
\[
\Trop{X}{\iota} \subset \trop(Y_\Sigma)
\]
is the image of the closed subset $\an{\iota(X)} \subset \an{Y_\Sigma}$ under $\trop$.
	
	Let $\mathfrak{a}\subset K[M]$ denote the ideal cutting out $\iota(X)\cap T$, where $T$ denotes the dense torus in $Y_{\Sigma}$. Every point $v \in N_{\RR}$ has an associated $R$-scheme
	\begin{equation}\label{equation: piece of model}
	\text{Spec}_{\ \!}R[M]^{v}\big/\big(\mathfrak{a}\cap R[M]^{v}\big),
	\end{equation}
where
	$$
	R[M]^{v}
	\ := \ 
	\left\{\sum_{u\in M}a_{u}\ \!\chi^{u}\in K[M] \, : \!\! 
	\begin{array}{c}
	\val(a_{u})+\langle u,v\rangle\ge0\ \mbox{ for all } u\in M
	\end{array} \! \!
	\right\}.
	$$
The special fiber of \eqref{equation: piece of model} is the {\em initial degeneration of $X$ at $v$}, denoted
	$
	\init_{v}(X);
	$
it is nonempty if and only if $v \in \Trop{X}{\iota}$. 

More generally, for any $w \in\trop(Y_{\Sigma})$, there is a unique cone $\sigma\in\Sigma$ such that $w$ lies in the tropicalization $\trop(O_\sigma)\cong N_\RR / \mathrm{span}(\sigma)$ of the open torus orbit $O_\sigma\subset Y_\Sigma$. In this case, $M_\sigma := \text{span}(\sigma)^\perp\cap M$ is dual to the lattice $\text{span}(\sigma)\cap N$, and we define
	$$
	R[M_\sigma]^{w}
	\  := \ 
	\left\{\sum_{u\in M_\sigma}a_{u}\ \!\chi^{u}\in K[M_\sigma] \, : \!\! 
	\begin{array}{c}
	\val(a_{u})+\langle u,w\rangle\ge0\ \mbox{ for all } u\in M_\sigma
	\end{array} \! \!
	\right\}.
	$$
The {\em initial degeneration of $X$ at $w$}, denoted $\init_{w}(\iota(X)\cap O_\sigma)$, is the special fiber of the scheme $\text{Spec}_{\ \!}R[M_\sigma]^{w}\big/\big(\mathfrak{a}\cap R[M_\sigma]^{w}\big)$. For each cone $\sigma \in \Sigma$, the intersection of $\Trop{X}{\iota}$ with $\trop(O_\sigma)$ is a finite polyhedral complex that parametrizes weight vectors on monomials in $M_\sigma$ such that $\init_w (\iota(X) \cap O_\sigma)$ 
is nonempty.   
For simplicity, given $\iota \colon  X\hookrightarrow Y_\Sigma$, $\sigma \in \Sigma$, and $w \in N_\RR / \mathrm{span}(\sigma)$, we sometimes write
$
\init_w(X)$ for $\init_w (\iota(X) \cap O_\sigma)$.
With this notation, 
\[
\Trop{X}{\iota} = \{ w \in \trop(Y_\Sigma) : \init_w (X) \neq \emptyset \}.
\]

\subsection{Admissible polyhedral covers}  \label{sec:admissible}
Let $\Gamma \subset \RR$ be the value group of $K$.  As in the previous section, we consider the toric variety $Y_\Sigma$ over $K$ associated to a fan $\Sigma$ in $N_\RR$.

Let $P \subset N_\RR$ be a polyhedron, the intersection of finitely many closed halfspaces.  Then the {\em recession cone} $\sigma_P$ is the closed polyhedral cone given by
\[
\sigma_P = \{ v \in N_\RR : v + P \subset P \}
\]
Equivalently, the cone $\sigma_P$ is obtained by taking the closure of the cone over $P \times \{ 1 \}$ in $N_\RR \times \RR$ and intersecting with $N_\RR \times \{ 0 \}$.

A polyhedron $P \subset N_\RR$  is $(\Gamma, \Sigma)$-\emph{admissible} if its recession cone $\sigma_P$ is in $\Sigma$, and $P$ itself can be expressed as an intersection of finitely many halfspaces
\[
P = \{ v \in N_\RR : \< u_i, v \> \geq \gamma_i, \mbox{ for } 1 \leq i \leq n \},
\]
with $u_i$ in the character lattice $M$ and $\gamma_i$ in $\Gamma$.  Note that any $(\Gamma, \Sigma)$-admissible polyhedron is pointed, i.e., its minimal faces have dimension zero, since the tail cone is in the fan $\Sigma$, which is a collection of pointed cones.

An \emph{extended $(\Gamma,\Sigma)$-admissible polyhedron} is the closure $\overline P$ in $\trop(Y_\Sigma)$ of a pointed $(\Gamma,\Sigma)$-admissible polyhedron $P$ in $N_\RR$.  An \emph{extended $(\Gamma,\Sigma)$-admissible polyhedral complex} $\Delta$ is a locally finite collection of extended $(\Gamma,\Sigma)$-admissible polyhedra, which we refer to as the \emph{faces} of the complex, whose intersections with $N_\RR$ form a polyhedral complex.  The \emph{support} $|\Delta|$ of $\Delta$ is the union of its faces, and $\Delta$ is \emph{complete} if $|\Delta| = \trop(Y_\Sigma)$.  Note that we require an extended $(\Gamma, \Sigma)$-admissible complex to be locally finite not only at points in the dense open subset $N_\RR \cap |\Delta|$, but at every point $w \in |\Delta|$.  See also Remark~\ref{rem:locallyfinite}, below.

\begin{remark}
Recall that $\Gamma$-admissible fans in $N_\RR \times \RR_{\geq 0}$, in the sense of \cite[\S7]{Gubler13}, correspond naturally and bijectively with normal toric varieties over the valuation ring $R \subset K$ \cite{GublerSoto15}.  The basic construction is recalled in \S\ref{section: Admissible polyhedral covers and Gubler models} below.  For now, note that extended $(\Gamma,\Sigma)$-admissible polyhedra correspond naturally and bijectively with $\Gamma$-admissible cones in $N_\RR \times \RR_{\geq 0}$ that meet $N_\RR \times \RR_{>0}$ and whose intersection with $N_\RR \times \{ 0 \}$ is a face of $\Sigma$.  This correspondence takes an extended $(\Gamma,\Sigma)$-admissible polyhedron $\overline P$ to the closure of the cone over $P \times \{1\}$.  Similarly, locally finite extended $(\Gamma,\Sigma)$-admissible polyhedral complexes correspond naturally and bijectively with the locally finite $\Gamma$-admissible fans in $N_\RR \times \RR_{\geq 0}$ whose restriction to $N_\RR \times \{0 \}$ is $\Sigma$.  The latter correspond, in turn, with locally finite type toric schemes over the valuation ring whose general fiber is $Y_\Sigma$.
\end{remark}

We will use the following lemma, on existence of complete complexes that simultaneously refine any finite collection of admissible polyhedra, in the proofs of Theorems~~\ref{theorem: cofinality of Gubler models} and \ref{theorem: 2}.  It is an immediate consequence of the following result of Coles and Friedenburg.

\begin{theorem*}[{\cite[Theorem~1.1]{ColesFriedenberg23}}]
Let $\Delta'$ be a finite extended $(\Gamma, \Sigma)$-admissible polyhedral complex. Then there is a complete and locally finite extended $(\Gamma, \Sigma)$-admissible polyhedral complex that contains $\Delta'$ as a subcomplex.
\end{theorem*}


\begin{lemma}\label{lem:refinements}
Let $\overline P_1, \dots, \overline P_n$ be extended $(\Gamma,\Sigma)$-admissible polyhedra.  Then there is a complete, locally finite extended $(\Gamma,\Sigma)$-admissible polyhedral complex $\Delta$ such that each $\overline P_i$ is a union of faces of $\Delta$, for $1 \leq i \leq n$.
\end{lemma}

\begin{proof}
For each $i$ there is a complete locally finite extended $(\Gamma,\Sigma)$-admissible complex $\Delta_i$ that contains $\overline P_i$ as a face, by \cite[Theorem~1.1]{ColesFriedenberg23}.  Then we can take $\Delta$ to be the smallest common refinement of $\Delta_1, \ldots, \Delta_n$.
\end{proof}

\subsection{Gubler models of toric varieties} \label{section: Admissible polyhedral covers and Gubler models}

If $\overline P \subset \trop(Y_\Sigma)$ is an extended $(\Gamma,\Sigma)$-admissible polyhedron, then $\trop^{-1}(\overline P)$ is a strictly affinoid analytic domain in $\an{Y_\Sigma}$, and is canonically realized as the Raynaud fiber of the formal completion of a flat $R$-scheme whose generic fiber is the affine open subvariety $U_{\sigma_P} \subset Y_\Sigma$, as we now explain.  Let $R[M]^{P}$ denote the $R$-algebra
	\begin{equation}\label{equation: definition of tilted algebra}
	R[M]^{P}
	\ =\ 
	\left\{\sum_{u\in M}a_{u}\ \!\chi^{u}\in K[M] \ : \!
	\begin{array}{c}
	\val(a_{u})+\langle u,v\rangle\ge0\ \mbox{ for all } v\in P \!\!
	\end{array}
	\right\}, 
	\end{equation}
and let $\mathscr{U}_{P}$ denote the $R$-scheme $\mathscr{U}_{P}=\Spec{R[M]^{P}}$. By \cite[Propositions 6.6 and 6.10]{Gubler13}, $\mathscr{U}_{P}$ is a normal scheme flat over $R$ with generic fiber $(\mathscr{U}_{P})_{K}:=\mathscr{U}_{P}\otimes_{R}K$ naturally isomorphic to the affine toric variety associated to the recession cone $\sigma_{P}$:
	$$
	(\mathscr{U}_{P})_{K}
	\ \cong\ 
	U_{\sigma_{P}}.
	$$
Let $\mathfrak{U}_{P}$ denote the formal $R$-scheme obtained as the completion of $\mathscr{U}_{P}$ along its special fiber. Since $K$ is algebraically closed, the value group $\Gamma$ is divisible, and hence the formal scheme $\mathfrak{U}_{P}$ is admissible \cite[Proposition~6.7]{Gubler13}.

	Let $\an{\mathfrak{U}_{P}}$ denote the Raynaud fiber of $\mathfrak{U}_{P}$. Then the Raynaud fiber $\an{\mathfrak{U}_{P}}$ is naturally identified with a strictly affinoid domain inside the Berkovich analytification $\an{U_{\sigma_{P}}}$ \cite[\S4.13]{Gubler13}. Specifically,
	$$
	\an{\mathfrak{U}_{P}}
	\ \cong\ 
	\big\{x\in\an{U_{\sigma_{P}}} : |f(x)|\le 1\ \mbox{for all } f\in K[S_{\sigma_{P}}] \big\}.
	$$
Applying \cite[Lemma 6.21]{Gubler13} to each orbit $O_{\sigma}\subset Y_{\Sigma}$, for $\sigma\in\Sigma$, we see that this strictly affinoid domain is the inverse image of $\overline P\subset\trop(Y_{\Sigma})$ under the tropicalization map, 
	\begin{equation}\label{equation: Raynaud fiber as inverse image of polyhedron}
	\an{\mathfrak{U}_{P}}=\trop^{-1}(\overline P).
	\end{equation}

These models of strictly affinoid domains associated to extended $(\Gamma,\Sigma)$-admissible polyhedra glue together naturally, as follows.  Each inclusion of a face of an extended $(\Gamma,\Sigma)$-admissible polyhedron $\overline Q \subset \overline P$ induces a Zariski open embedding of $R$-schemes $\mathscr{U}_{Q}\subset \mathscr{U}_{P}$ and a Zariski open embedding of formal $R$-schemes $\mathfrak{U}_{Q}\subset \mathfrak{U}_{P}$.  If $\Delta$ is an extended $(\Gamma,\Sigma)$-admissible polyhedral complex, gluing along these open embeddings naturally produces a flat $R$-scheme $\mathscr{Y}_{\Delta}$ and an admissible formal $R$-scheme $\mathfrak{Y}_\Delta$, respectively.  By construction, $\mathfrak{Y}_{\Delta}$ is the formal completion of $\mathscr{Y}_{\Delta}$ along its special fiber.  Furthermore, the Raynaud fiber $\an{\mathfrak{Y}_{\Delta}}$ is the preimage of the support $|\Delta| \subset \trop(Y_\Sigma)$ under the tropicalization map, i.e.,
	\begin{equation} \label{eq:isom}
	\an{\mathfrak{Y}_{\Delta}} = \mathrm{Trop}^{-1}(|\Delta|).
	\end{equation}

\begin{remark} \label{rem:locallyfinite}
In \cite{Gubler13}, such algebraic and formal models are considered only for finite complexes $\Delta$.  Even for an arbitrary (infinite) complex $\Delta$, one may construct the affine $R$-schemes $\mathscr{U_P}$, glue along the open embeddings corresponding to inclusions of faces to construct $\mathscr{Y}_\Delta$, and formally complete along the special fiber to obtain $\mathfrak{Y}_\Delta$.  Note, however, that if some face of the complex is contained in infinitely many other faces, then the resulting formal scheme is not paracompact, and hence does not have a Raynaud fiber in the category of Berkovich spaces.  Moreover, even when every face is contained in only finitely many other faces, one must take care in comparing the Raynaud fiber $\an{\mathfrak{Y}_{\Delta}}$ with $\mathrm{Trop}^{-1}(|\Delta|)$; indeed, Example~\ref{ex:not-loc-finite} shows that these two analytic spaces are not isomorphic in general when $\Delta$ is not locally finite.  Nevertheless, the standard constructions of $\an{\mathfrak{Y}_{\Delta}}$, its natural map to $\mathrm{Trop}^{-1}(|\Delta|)$, and the proof that this map is an isomorphism when $\Delta$ is finite are all local; they extend verbatim to the case where $\Delta$ is locally finite.
\end{remark}

\begin{example} \label{ex:not-loc-finite}
We briefly sketch an example of a $(\Gamma,\Sigma)$-admissible complex that is not locally finite, for which the analogue of \eqref{eq:isom} does not hold.  Let $\Gamma = \QQ$, $N = \ZZ$, and $\Sigma = \{ 0 \}$.  Consider the complex $\Delta$ whose maximal faces are the point $0$ together with the intervals $ \big[\frac{1}{n+1},\frac{1}{n} \big]$ for positive integers $n$.  Then $|\Delta| = [0,1]$, but $\Delta$ is not locally finite at $0$.  The Raynaud fiber $\an{\mathfrak{Y}_{\Delta}}$ is the disjoint union of $\mathrm{Trop}^{-1}(0)$ and $\mathrm{Trop}^{-1}((0,1])$, which is disconnected, and not isomorphic to $\mathrm{Trop}^{-1}(|\Delta|)$.
\end{example}

\subsection{Gubler models of closed subvarieties}\label{section: Gubler models of closed subvarieties}
	Let $\iota \colon X \hookrightarrow Y_\Sigma$ be the inclusion of a closed subscheme, and let $\Delta$ be an extended $(\Gamma,\Sigma)$-admissible polyhedral complex in $\trop(Y_{\Sigma})$.  Then, for each face $\overline P$ of $\Delta$, we obtain an $R$-model $\mathscr{X}_{P}$ of $X\cap U_{\sigma_P}$ as the closure of $X$ in the toric scheme $\mathscr{U}_P$.  Gluing along the inclusions of faces in $\Delta$, we obtain an $R$-model $\mathscr{X}_{\Delta}$ of $X$, which we refer to as the ({\em algebraic}) {\em Gubler model of $X$ associated to the pair $(\iota,\Delta)$}.  See also \cite{ColesFriedenberg23b} for further discussion of such formal models.
	
	Completing each $R$-scheme $\mathscr{X}_{P}$ along its special fiber, we obtain an admissible formal $R$-scheme $\mathfrak{X}_{P}$. Gluing these formal schemes along the inclusions in $\Delta$, we obtain an admissible formal $R$-scheme $\mathfrak{X}_{\Delta}$ isomorphic to the completion of $\mathscr{X}_{\Delta}$ along its special fiber, whose Raynaud fiber $\an{\mathfrak{X}_\Delta}$ is the preimage of $|\Delta|$ in $\an{X}$.  In particular, $\mathfrak{X}_\Delta$ is a formal model of $\an{X}$ if and only if $|\Delta|$ contains $\Trop{X}{\iota}$. In this case, we say that $\Delta$ \emph{covers} $\Trop{X}{\iota}$ and refer to $\mathfrak{X}_{\Delta}$ as the {\em formal Gubler model of $\an{X}$ associated to the pair $(\iota,\Delta)$}. When we want to stress the role of the closed embedding $\iota\colon X \hookrightarrow Y_\Sigma$ in the construction of the formal Gubler model $\mathfrak{X}_{\Delta}$, we write $\mathfrak{X}_{(\iota,\Delta)}$.

\subsection{Adic tropicalization as an inverse limit of formal models}\label{subsection: adic tropicalization as an inverse limit of formal models}

	Let $\Delta$ and $\Delta'$ be extended $(\Gamma,\Sigma)$-admissible polyhedral complexes that cover $\Trop{X}{\iota}$.  We say that $\Delta'$ \emph{refines} $\Delta$ if each face $\overline P'$ of $\Delta'$ is contained in some face $\overline P$ of $\Delta$.  The induced maps $\mathfrak{X}_{P'} \rightarrow \mathfrak{X}_P$ glue to produce a morphism of formal $R$-schemes $\mathfrak{X}_{\Delta'}\lra\mathfrak{X}_{\Delta}$.  This gives a functor from the category of extended $(\Gamma, \Sigma)$-admissible complexes that cover $\Trop{X}{\iota}$, in which the morphisms are refinements, to the category of admissible formal models of $\an{X}$.

\begin{definition}
	The {\em adic tropicalization} of the closed embedding $\iota\colon X \hookrightarrow Y_\Sigma$ is the locally topologically ringed space 
	$$
	\big(\mathfrak{Trop}(X,\iota),\mathscr{O}_{\mathfrak{Trop}(X,\iota)}\big)
	\ :=\ 
	\varprojlim (\mathfrak{X}_{\Delta},\mathscr{O}_{\mathfrak{X}_{\Delta}}),
	$$
	where the limit is taken over all models associated to extended $(\Gamma, \Sigma)$-admissible polyhedral complexes that cover $\Trop{X}{\iota}$ and all morphisms induced by refinements. 
	\end{definition}
	
\noindent	When no confusion seems possible, we denote both the adic tropicalization of $\iota$ and its underlying topological space simply by $\mathfrak{Trop}(X,\iota)$.  See \cite[Remark~2.2.4]{Foster16} for the existence of limits in the category of locally topologically ringed spaces.

\begin{remark}\label{remark: on adic tropicalization}
	For a detailed example of the adic tropicalization of a line in $\mathbb{P}^{2}$, with accompanying figures, see \cite[Example~3.5.6]{Foster16}.
\end{remark}
	
	Consider the \emph{category of toric embeddings of $X$}, whose objects are closed embeddings $\iota \colon X\hookrightarrow Y_{\Sigma}$ into toric varieties, and whose morphisms are commutative diagrams
	$$
	\begin{aligned}
	\begin{xy}
	(0,0)*+{X}="1";
	(15,7)*+{Y_{\Sigma}}="2";
	(15,-7)*+{Y_{\Sigma'}}="3";
	{\ar@{^{(}->}^{\iota} "1"; "2"};
	{\ar@{_{(}->}_{\iota'} "1"; "3"};
	{\ar@{->}^{f} "2"; "3"};
	\end{xy}
	\end{aligned}
	$$
induced by a toric morphism $\phi\colon\Sigma\lra\Sigma'$ of fans.  Given such a diagram, let $\Delta'$ be an extended  $(\Gamma, \Sigma)$-admissible polyhedral complex that covers $\Trop{X}{\iota'}$.  Note that we can refine any extended $(\Gamma, \Sigma)$-admissible polyhedral complex $\Delta$ that covers $\Trop{X}{\iota}$ to obtain a cover $\Delta^{+}$ such that the induced morphism $\trop(f)\colon\trop(Y_{\Sigma})\lra\trop(Y_{\Sigma'})$ maps each $\overline P \in \Delta^{+}$ into some $\overline P' \in\Delta'$. There are then induced morphisms of algebraic and admissible formal $R$-schemes
\[
\mathscr{X}_{P} \rightarrow \mathscr{X}_{P'} \mbox{ \ \ and \ \ } \mathfrak{X}_{P} \rightarrow \mathfrak{X}_{P'},
\]
which glue to give morphisms of algebraic and formal Gubler models
\[
\mathscr{X}_{\Delta} \rightarrow \mathscr{X}_{\Delta'} \mbox{ \ \ and \ \ } \mathfrak{X}_{\Delta} \rightarrow \mathfrak{X}_{\Delta'}.
\]

From this it follows that each morphism $f$ of toric embeddings induces a morphism of adic tropicalizations, 
$$
\mathfrak{Trop}(X,f)\colon\mathfrak{Trop}(X,\iota)\lra\mathfrak{Trop}(X,\iota'),
$$
in the category of locally topologically ringed spaces; hence, adic tropicalization is a functor from toric embeddings to locally topologically ringed spaces.

\subsection{Adic tropicalization as a union of initial degenerations}\label{subsection: adic tropicalization as a union of initial degenerations}
We now show that the disjoint union of the initial degenerations $\init_w(X)$, for $w \in \trop(X, \iota)$, is naturally identified with the underlying set of the adic tropicalization $\fTrop{X}{\iota}$.
	
\begin{lemma}\label{lemma: initial degeneration as an inverse limit}
	For each point $w\in\Trop{X}{\iota}$, not necessarily $\Gamma$-rational, there is a natural isomorphism of $k$-schemes
	$$
	\init_{w}(X)
	\xrightarrow{\ \sim\ }
	\varprojlim_{P\ni w}(\mathfrak{X}_{P})_{s},
	$$
	where $P$ ranges over all $(\Gamma,\Sigma)$-admissible polyhedra that contain $w$.
\end{lemma}
\begin{proof}
	Consider first the special case where $w$ is in the dense open subset $N_\RR \subset \mathrm{Trop}(Y_\Sigma)$.  For each $(\Gamma,\Sigma)$-admissible polyhedron $P$ that contains $w$, we have $R[M]^{P}\subset R[M]^w$. Furthermore, from \eqref{equation: definition of tilted algebra} we see that, for each $f \in K[M]$ there is a $(\Gamma,\Sigma)$-polyhedron $P_f$ (possibly empty) such that $f \in R[M]^{w}$ if and only if $w \in P_f$.  It follows that $R[M]^w = \bigcup_{P \ni w} R[M]^P$.  Passing to the special fibers of the associated $R$-schemes, we conclude that $\init_w{X} \cong \varprojlim_{P \ni w} (\mathfrak{X}_{P})_s$, as required.
	
	For the general case, consider $w\in\trop(O_\sigma)$, for some $\sigma\in \Sigma$. Let $\overline P$ be an extended $(\Gamma,\Sigma)$-admissible polyhedron that contains $w$.  Then the recession cone of $P$  contains $\sigma$, and $\overline{P}\cap\text{Trop}(O_\sigma)$ is the image of $P$ under the projection $N_\RR\longrightarrow N_\RR/\text{span}(\sigma)$. Because each linear function in $u \in M_\sigma$ is constant on the fibers of this projection, we obtain inclusions $R[M_\sigma]^{\overline{P}\cap\trop(O_\sigma)}\hookrightarrow R[M]^{P}$, and 
	\begin{equation}\label{equation: needed isomorphism of unions*}
	\bigcup_{\overline{P}\ni w}R[M_\sigma]^{\overline{P}\cap\trop(O_\sigma)}
	\ \xymatrix{{}\ar@{^{(}->}[r]&{}}
	\bigcup_{\overline{P}\ni w}\!R[M]^{P}.
	\end{equation}
We must show that \eqref{equation: needed isomorphism of unions*} becomes an isomorphism over the residue field $k$.  As a first step, suppose $a_u \chi^u$ is a nonzero monomial in $\bigcup_{\overline{P}\ni w}\!R[M]^{P}$ and $u \not \in M_\sigma$.  We claim that $a_u \chi^u$ vanishes modulo the maximal ideal $\mathfrak{m} \subset R$.  To see this, note that the linear function $u$ must be strictly positive on the interior of $\sigma$, which contains a $\Gamma$-rational point $v'$.  Let $P' = P+v'$.  Then $w \in \overline{P}'$ and $a_u \chi^u \in \mathfrak{m} R[M]^{P'}$.  Hence $a_u \chi^u$ vanishes modulo $\mathfrak{m}$, as claimed.  The remainder of the argument is similar to the previous case.
\end{proof}

\begin{proposition}\label{proposition: adic trop as exploded top}
	There is a natural bijection
	\begin{equation}\label{equation: exploded trop identification map}
	\bigsqcup_{w\in\trop(X,\iota)}\!\!\!\!\!\!\big|\init_{w}(X)\big|
	\xrightarrow{\ \sim\ }
	\big|\mathfrak{Trop}(X,\iota)\big|
	\end{equation}
between the set underlying the adic tropicalization $\fTrop{X}{\iota}$ and the disjoint union of all initial degenerations of $X$ at all points on $\trop(X,\iota)$.
\end{proposition}
\begin{proof}
	To construct the map \eqref{equation: exploded trop identification map}, it suffices to construct a map $|\init_{w}(X)|\lra\varprojlim_{\Delta} \big|\mathfrak{X}_{\Delta}\big|$, where the limit is over $(\Gamma,\Sigma)$-admissible covers of $\Trop{X}{\iota}$. Each such cover contains a face $\overline P$ that contains $w$, and the existence of this map then follows from Lemma~\ref{lemma: initial degeneration as an inverse limit}.
	
		By \cite[Proposition 8.8]{Gubler13}, if $\overline P$ and $\overline Q$ are disjoint faces of $\Delta$, then $\mathfrak{X}_P$ and $\mathfrak{X}_Q$ are disjoint in $\mathfrak{X}_\Delta$.  It follows that the map \eqref{equation: exploded trop identification map} is injective.  To see that \eqref{equation: exploded trop identification map} is surjective, observe that again by \cite[Proposition 8.8]{Gubler13} any point $p$ in the inverse limit $\varprojlim_{\Delta} \big|\mathfrak{X}_{\Delta} \big|$ projects into $\mathfrak{X}_{P_{i}}$ for some nested decreasing sequence of admissible polyhedra $P_{i}$ in $\trop(Y_{\Sigma})$ such that $\bigcap P_{i}$ is the single point $w$. Then $p$ is in the image of $\big|\init_{w}(X)\big|$.
\end{proof}


\section{The adic limit theorem}\label{Section-Main Theorem}


In this section, we prove Theorem~\ref{thm:adiclimit}, showing that the adic analytification of a subscheme of a toric variety is recovered as the inverse limit of the adic tropicalizations of any system of toric embeddings that satisfies the condition ($\dagger$) from the introduction.  The proof is not as direct as those of the tropical limit theorems in \cite{analytification, limits}.  We first prove a strong cofinality statement for Gubler models of a projective scheme (Theorem~\ref{theorem: 2}),  
and then deduce that adic analytifications of certain strictly affinoid  domains in analytifications of affine schemes can be obtained as inverse limits of algebraizable formal models (Corollary~\ref{cor:special-affinoid}). 
Theorem~\ref{thm:adiclimit} then follows easily.

\begin{subsection}{Cofinality of Gubler models for strictly affinoid domains}
As in the previous sections, we let $X$ denote a separated scheme of finite type over $K$. 
Recall that a formal scheme $\mathfrak{X}$ is \emph{algebraizable} if it is isomorphic to the formal completion of a flat $R$-scheme of finite type.  By construction, any formal Gubler model is algebraizable.

\begin{lemma} \label{lem:algebraizable}
If $X$ is projective, then the algebraizable formal models are cofinal in the inverse system of all formal models of $\an{X}$.
\end{lemma}

\begin{proof}
Let $\Xc$ be an arbitrary formal model of $\an{X}$.  We must show that there is an algebraizable formal model that dominates $\Xc$.  Fix an embedding  of $X$ in the projective space $\mathbb{P}^{n}_K$. Then the closure $\mathscr{X}$ of $X$ in $\mathbb{P}^n_R$ is flat and hence of finite presentation over $R$ \cite[Corollary~3.4.7]{RaynaudGruson71}.  Since $\mathscr{X}$ is of finite presentation, its formal completion $\Xc'$ is admissible.  Because $X$ is projective, the adic space $\ad{X}$ is quasi-compact, and therefore every admissible formal model of $\ad{X}$ is {\em quasi-paracompact} in the sense of \cite[\S8.2, Definition 12]{Bosch14}. By Raynaud's theorem \cite[\S8.4, Theorem 3]{Bosch14}, this implies that there is an admissible formal blowup $\Xc'' \rightarrow \Xc'$ such that $\Xc''$ admits a morphism to $\Xc$.  The center of this admissible formal blowup is a coherent ideal sheaf on $\Xc'$ and, since $\Xc'$ is projective over $\Spf R$, this coherent sheaf is algebraizable \cite[Theorem~2.13.8]{Abbes10}.  Then $\Xc''$ is an algebraizable formal model of $\an{X}$ that dominates $\Xc$, and the lemma follows.
\end{proof}

We now state and prove a technical theorem about the Gubler models $\Xc_{\Delta}$ of projective schemes.

\begin{definition}\label{definition: adapted model}
	We will say that a formal Gubler model $\Xc_\Delta$ of $X$ is \emph{adapted} to an analytic domain $V\subset \an{X}$ if there is a subcomplex $\Delta' \subset \Delta$ such that $V$ is the Raynaud fiber $\an{\Xc_{\Delta'}}\subset \an{\Xc_{\Delta}}$.
\end{definition}

\noindent Equivalently, $\Xc_\Delta$ is adapted to $V$ if $V = \mathrm{Trop}^{-1}(|\Delta'|)$ for some subcomplex $\Delta'$ of $\Delta$.

\begin{theorem}\label{theorem: 2}
Let $X$ be projective over $K$, let $U \subset X$ be an affine open subscheme, and let 
$\{g_1, \ldots, g_n\} \subset K[U]$ be a set of generators for the coordinate ring.  Consider the strictly affinoid domain
\[
V = \big\{ x \in \an{U} \, : \, |g_\ell(x)| \leq 1, \mbox{ for }  1 \leq \ell \leq n\big\}.
\]
Then the formal Gubler models $\Xc_{(\iota,\Delta)}$ adapted to $V$ such that $U$ is the preimage under $\iota$ of a torus invariant affine open subvariety are cofinal in the inverse system of all formal models of $\an{X}$.
\end{theorem}

\begin{proof}
By Lemma~\ref{lem:algebraizable}, the algebraizable models of $X$ are cofinal in the inverse system of all formal models. Thus it suffices to show that every algebraizable model is dominated by a formal Gubler model $\Xc_{(\iota,\Delta)}$ that is adapted to $V$, such that $U$ is the preimage under $\iota$ of a torus invariant affine open subvariety.

Fix an affine open cover $X = U_1 \cup \cdots \cup U_r$, with $U = U_1$.  Let $\mathscr{X}$ be a flat and proper $R$-scheme with generic fiber $X$. Note that $\mathscr{X}$ is finitely presented over $R$, by \cite[Corollary~3.4.7]{RaynaudGruson71}.  Cover $\mathscr{X}$ by finitely many affine opens.  After refinement, we may assume that this cover consists of affine opens labeled $\mathscr{U}_{ij}$ such that, for each $i$ and $j$, the generic fiber $\mathscr{U}_{ij} \times \Spec K$ is a distinguished affine open $D(f_{ij}) \subset U_i$ for some $f_{ij} \in K[U_i]$. For each $i$ and $j$, choose a presentation
	\[
	\mathscr{U}_{ij}\ =\ \Spec R[x_{ij1}, \ldots, x_{ijs}] \big/ \mathfrak{a}_{ij},
	\]
where $x_{ijk} = y_{ijk}\big/f_{ij}^{m_k}$, for some $y_{ijk} \in K[U_i]$ and $m_k \in \ZZ_{\geq 0}$. 

By \cite[Theorem~4.2]{limits}, the inverse system of all toric embeddings of $X$ satisfies ($\dagger$) with respect to any affine open cover, so we can choose a closed embedding $\iota\colon X \hookrightarrow Y_\Sigma$ such that 
\begin{itemize}
\item [{\bf (i)}] each $U_i$ is the preimage of a torus invariant affine open $U_{\sigma_i}$, for $1 \leq i \leq r$,
\item [{\bf (ii)}] each of the functions $f_{ij}$ and $y_{ijk}$ is the pullback of a character that is regular on $U_{\sigma_i}$, and
\item [{\bf (iii)}]each of the functions $g_\ell$ appearing in the statement of the theorem is the pullback of a character that is regular on $U_{\sigma_1}$, for $1 \leq \ell \leq n$.
\end{itemize}
Let $\Xc$ be the formal completion of $\mathscr{X}$. We will construct an admissible polyhedral cover $\Delta$ of $\Trop{X}{\iota}$ such that the associated Gubler model $\Xc_\Delta$ is adapted to $V$ and dominates $\Xc$. The proof involves finding admissible polyhedra $P$ and $P_{ij}$ such that $V = \an{\Xc_P}$ and $\an{(\widehat{\mathscr{U}}_{ij})} = \an{\Xc_{P_{ij}}}$ and then applying Lemma~\ref{lem:refinements}. The details are as follows.

Let $M$ be the character lattice of the dense torus in $Y_\Sigma$.  Choose $u_1, \ldots, u_n$ in $M$ such that $g_\ell$ is the pullback of the character $\chi^{u_\ell}$, which is regular on $U_{\sigma_1}$.  If these do not generate $K[U_{\sigma_1}]$, then choose additional characters $\chi^{u_{n+1}}, \ldots, \chi^{u_{n'}}$ so that $\{ \chi^{u_i} \, : \, 1 \leq i \leq n' \}$ does generate.  Let $a_i$ be the minimum of the continuous function $\val \chi^{u_i}$ on the compact space $V$.  Let $\overline P$ be the extended $(\Gamma,\Sigma)$-admissible polyhedron in $\trop(Y_\Sigma)$ given by the closure of
\[
P = \{ v \in N_\RR \, : \, \langle u_i, v \rangle \geq a_i \mbox{ for } 1 \leq i \leq n' \}.
\]
Then $V = \trop^{-1}(\overline P)$, by construction.  By \eqref{equation: Raynaud fiber as inverse image of polyhedron}, it follows that $V$ is the Raynaud fiber $\an{\Xc_P}$.

Next, choose $u_{ij}$ and $u_{ijk}$ in $M$ such that $f_{ij}$ and $y_{ijk}$ are the pullbacks of the characters $\chi^{u_{ij}}$ and $\chi^{u_{ijk}}$, respectively, each of which is regular on $U_{\sigma_i}$.  Recall that the distinguished affine open $D(f_{ij}) \subset U_i$ is the general fiber of $\mathscr{U}_{ij}$, and hence is the preimage of the torus invariant affine open $U_{\sigma_{ij}} \subset U_{\sigma_i}$ on which $\chi^{u_{ij}}$ is invertible.  In particular, the characters $\chi^{u_{ijk} - m_k u_{ij}}$ are regular on $U_{\sigma_{ij}}$.  As in the previous paragraph, we can choose additional characters to generate $K[U_{\sigma_{ij}}]$ and the valuations of all of these characters achieve their minima on the compact strictly affinoid domain $\an{(\widehat{\mathscr{U}}_{ij})}$.  Let $\overline P_{ij}$ be the closure in $\trop(Y_\Sigma)$ of the polyhedron in $N_\RR$ on which the linear functions corresponding to each of these characters is bounded below by the respective minima of valuations on $\an{(\widehat{\mathscr{U}}_{ij})}$.  Then $P_{ij}$ is $(\Gamma,\Sigma)$-admissible by construction, and there is a natural morphism
\[
\varphi_{ij} \colon \Xc_{P_{ij}} \rightarrow \Xc,
\]
inducing an identification $\an{\Xc_{P_{ij}}} \xrightarrow{\sim} \an{(\widehat{\mathscr{U}}_{ij})}$ on the Raynaud fibers.  Since $\mathscr{X}$ is proper over $R$, we have that $\{\an{(\widehat{\mathscr{U}}_{ij})}\}$ covers $\an{X}$ and hence $\{\overline P_{ij}\}$ covers $\Trop{X}{\iota}$.

By Lemma~\ref{lem:refinements}, there is an extended $(\Gamma,\Sigma)$-admissible polyhedral complex $\Delta$ such that $\overline P$ and each of the $\{\overline P_{ij}\}$ is a union of faces of $\Delta$.  Moreover, since the $\{ \overline P_{ij} \}$ cover $\Trop{X}{\iota}$, we may assume that each face of $\Delta$ is contained in some $\overline P_{ij}$, without changing the Gubler model $\Xc_\Delta$.  Then each inclusion $\overline Q \subset \overline P_{ij}$, for $Q \in \Delta$ induces morphisms of formal schemes
\[
\Xc_Q \longrightarrow \Xc_{P_{ij}} \xrightarrow{\varphi_{ij}} \Xc,
\]
which glue together to give $\Xc_\Delta \rightarrow \Xc$, a morphism of formal models of $\an{X}$, as required. Here, we use that the recession cone of each  $P_{ij}$ is in $\Sigma$ to ensure that the generic fiber of $\Xc_\Delta$ is $\an{X}$, and not some modification associated to a toric blowup of $Y_\Sigma$.
\end{proof}

\begin{corollary}  \label{cor:special-affinoid}
Let $U$ be an affine scheme of finite type over $K$, with $\{ g_1, \ldots, g_n \}$ a generating set for the coordinate ring $K[U]$.  Let $V \subset \an{U}$ be the strictly affinoid domain 
\[
V = \{ x \in \an{U} \, : \, |g_\ell(x)| \leq 1 \mbox{ \ for \ } 1 \leq \ell \leq n\},
\]
and let $\cS_V$ be the inverse system consisting of all formal models of $V$ of the form $\mathfrak{U}_{(\iota, \Delta)}$, where $\iota$ is a closed embedding of $U$ into an affine toric variety.  Then the pair $(\ad{V},\mathscr{O}^{\circ}_{\ad{V}})$ is naturally isomorphic to $\varprojlim_{\cS_V} (\mathfrak{U}_{(\iota, \Delta)},\mathscr{O}^{\circ}_{\ad{V}})$.
\end{corollary}

\begin{proof}
Choose a projective compactification $X$ of $U$.  Recall from \S\ref{specializations of adic spaces} that $\ad{X}$ is the projective limit of the inverse system of formal models of $\an{X}$.  By Theorem~\ref{theorem: 2}, the formal Gubler models $\Xc_{(\iota,\Delta)}$ that are adapted to $V$ and such that $U$ is the preimage under $\iota$ of a torus invariant affine open subvariety $U_\sigma$ are cofinal in the inverse system of all formal models of $\an{X}$.  Hence $\ad{V}$ is isomorphic to the inverse limit of formal models $\mathfrak{X}_{\Delta'}$ of $V$ that appear as formal Zariski opens $\mathfrak{X}_{\Delta'}\subset\Xc_{(\iota,\Delta)}$ arising from subcomplexes $\Delta' \subset \Delta$, as in Definition \ref{definition: adapted model}. We claim that each of these formal models is dominated by a model in $\cS_V$.  To see this, choose such a model. A priori, some of the faces of $\Delta$ may have recession cones that are not faces of $\sigma$, but are rather cones in some larger fan $\Sigma$, associated to the toric variety in which $X$ is embedded. However, using the fact that $\trop(V)$ is a compact subset of $\trop(U_\sigma)$, we may refine $\Delta$ so that $\trop(V)$ is covered by a $(\Gamma,\sigma)$-admissible subcomplex $\Delta''$ of the induced refinement of $\Delta'$.  Then $\Xc_{\Delta''}$ is in $\cS_V$ and dominates $\Xc_{\Delta'}$, as required.
\end{proof}

\begin{proof}[{\it Proof of Theorem \ref{thm:adiclimit}}]
	Let $X$ be a separated finite type $K$-scheme, and let $\Sc$ be a system of toric embeddings of $X$ satisfying condition ($\dagger$) of \S\ref{section: introduction} with respect to an affine open cover $X = U_{1} \cup \cdots \cup U_{r}$. We want to show that the induced map
	\begin{equation}\label{equation: canonical map for proof of admic limit theorem}
	(\ \!\ad{X},\mathscr{O}_{X^{\text{ad}}}^{\circ}) \lra \varprojlim_{\mathcal S} \big(\ \!\fTrop{X}{\iota},\ \!\mathscr{O}_{\fTrop{X}{\iota}}\ \!\big)
	\end{equation}
is an isomorphism of locally topologically ringed spaces.
	
	Consider a single affine open $U = U_{i}$ in our chosen cover. Choose a generating set $g_1, \ldots, g_n$ for the coordinate ring $K[U_i]$.  Let $V \subset \an{U}$ be the strictly affinoid domain
	\[
V = \big\{ x \in \an{U} \, : \, |g_\ell(x)| \leq 1, \mbox{ for }  1 \leq \ell \leq n\big\}.
\]
By Corollary~\ref{cor:special-affinoid}, the ringed space $\big(\ad{V},\mathscr{O}^{\circ}_{\ad{V}}\big)$ is naturally isomorphic to the inverse limit of its formal models arising as $\mathfrak{U}_{(\jmath, \Delta_0)}$, where $\jmath \colon U \hookrightarrow U_\tau$ is a closed embedding into an affine toric variety and $\Delta_0$ is an extended $(\Gamma, \tau)$-admissible polyhedral complex such that $\trop^{-1}(|\Delta_0|) = V$.  (Here, we use $\tau$ also to denote the fan consisting of its faces.)  We fix one such model $\mathfrak{U}_{(\jmath, \Delta_0)}$.

Replacing the generating set $\{ g_1, \ldots, g_n\}$ in the description above by $\{a g_1, \ldots, ag_n \}$, for $a \in K^*$ with $|a| \ll 1$, we can cover $\an{U}$ by strictly affinoid domains of this form.
Since the analogous statement holds for each of the affine opens $U_1, \ldots, U_r$ and these cover $X$, to prove the theorem it will suffice to show that there is a Gubler model of $X$ associated to an embedding $\iota \in \cS$ that is adapted to $V$, and such that the induced formal model of $V$ dominates $\mathfrak{U}_{(\jmath, \Delta_0)}$.

Let $h_1, \ldots, h_s$ be the respective pullbacks to $U$ of characters $\chi^{u_1}, \ldots, \chi^{u_s}$ that generate $K[U_\tau]$.  By condition ($\dagger$) there is an embedding $\iota \colon X \hookrightarrow Y_\Sigma$ in  $\cS$ such that $U = \iota^{-1}(U_\sigma)$ for some torus invariant affine open $U_\sigma \subset Y_\Sigma$ and $h_1, \ldots, h_s$ are the pullbacks of characters that are regular on $U_\sigma$.  The choice of such characters that are regular on $U_\sigma$ induces a toric morphism $\pi \colon U_\sigma \rightarrow U_\tau$ that commutes with the embeddings of $U$.

Let $\overline P_0$ be a face of $\Delta_0$.  The preimage $\overline{Q}_0 = \trop(\pi)^{-1}(\overline P_0)$ in $\trop(U_\sigma)$ is not necessarily $(\Gamma, \sigma)$-admissible, since the recession cone of $Q_0$ is not necessarily a face of $\sigma$.  Note, however, that the recession cone of $Q_0$ is the preimage of a face of $\tau$, and hence cut out by some supporting hyperplanes of the preimage of $\tau$.  Since $\sigma$ maps into $\tau$, these hyperplanes pullback to supporting hyperplanes of $\sigma$.  Moreover, since $V$ is compact, the continuous functions $\val \chi^{u_i}$ achieve their respective minima $a_i$.  Let $Q$ be the polyhedron obtained by intersecting $Q_0$ with the halfspaces $\langle u_i, v \rangle \geq a_i$, for $1 \leq i \leq s$.  Thus, the recession cone of $Q$ is the intersection of $\sigma$ with some of its supporting hyperplanes, and hence $Q$ is $(\Gamma, \sigma)$-admissible. Then  
$\trop^{-1}(\overline Q) \subset \an{X}$ is equal to $\trop^{-1}(\overline P)$, and there is a naturally induced morphism of formal models $\mathfrak{U}_{Q} \rightarrow \mathfrak{U}_P$.

Performing the above procedure for each face of $\Delta_0$, we arrive at a collection of extended $(\Gamma, \Sigma)$-admissible polyhedra $\overline Q_1, \ldots, \overline Q_r$ with morphisms of formal models $\mathfrak{U}_{Q_i} \rightarrow \mathfrak{U}_{\Delta_0}$ and
\[
V = \trop^{-1}(\overline Q_1 \cup \cdots \cup \overline Q_r).
\]
By Lemma~\ref{lem:refinements}, there is an extended $(\Gamma, \Sigma)$-admissible complex $\Delta$ that covers $\Trop{X}{\iota}$ such that each $\overline Q_i$ is a union of faces of $\Delta$.  Then the Gubler model $\mathfrak X_\Delta$ is adapted to $V$ and the induced formal model dominates $\mathfrak{U}_{\Delta_0}$, as required.
\end{proof}	

\subsection{Gubler models of compact analytic domains}  We now extend our results on cofinality of Gubler models from projective schemes to compact analytic domains in subschemes of toric varieties, proving Theorem~\ref{theorem: cofinality of Gubler models}.  We begin with an algebraic lemma.
	
\begin{lemma}\label{lemma: presence of inverse}
	Let $\mathfrak{X}$ be an admissible formal $R$-scheme and let $g \in \Gamma(\mathfrak{X}, \mathscr{O}_{\mathfrak X})$ be a function whose restriction to the adic fiber is invertible in $\Gamma(\ad{\mathfrak{X}},\mathscr{O}^{\circ}_{\!\!X^{\text{ad}}})$.  Then $g^{-1} \in \mathscr{O}_ {\mathfrak{X}}$.
	\end{lemma}

\begin{proof}
	It suffices to consider the case where $\mathfrak{X} = \Spf A$ is affine.  Since the restriction of $g$ to the adic fiber is invertible in $\Gamma(\ad{\mathfrak{X}},\mathscr{O}^{\circ}_{\!\!X^{\text{ad}}})$, its inverse must be power bounded, i.e., we have $g^{-1} \in A_K^\circ$.  It remains to show that $g^{-1}$ is in the subring $A \subset A_K^\circ$.
	
	By \cite[Proposition 2.12]{GublerRabinoffWerner17}, the power bounded ring $A_K^\circ$ is an integral extension of $A$. Therefore, $g^{-1}$ must satisfy an identity $(g^{-1})^{n}+c_{n-1}(g^{-1})^{n-1}+\cdots+c_{0}=0$ with coefficients $c_i \in A$.  Multiplying by $g^{n-1}$, gives $g^{-1}=-c_{n-1}-\cdots-c_{0}g^{n-1}$.  In particular, $g^{-1}$ is in $A$, as required.
\end{proof}

This lemma allows us to construct formal Gubler models that are adapted to certain rational domains, as follows.

\begin{proposition}\label{proposition: Theorem 1.3 for distinguished opens}
	Let $X$ be a separated $K$-scheme of finite type, let $\mathcal{S}$ be an inverse system of toric embeddings of $X$ satisfying $(\dagger)$, and let $\mathfrak{X}_{\Delta}$ be the formal Gubler model associated to a pair $(\iota,\Delta)$, for $\iota$ in $\mathcal{S}$.  Let $P \in \Delta$ be an admissible polyhedron and let $\mathfrak{U}_P \big \langle \frac{1}{g} \big \rangle$ be the distinguished formal open subscheme of $\mathfrak{U}_P$ associated to some $g \in \widehat{R[M]^P}$.  Then there is a pair $(\iota', \Delta')$ with $\iota'$ in $\mathcal{S}$ whose associated formal Gubler model $\mathfrak{X}_{\Delta'}$ is adapted to $\an{\mathfrak{U}_P \big \langle \frac{1}{g} \big \rangle}$ and dominates $\mathfrak{X}_\Delta$.
\end{proposition}

\begin{proof}
	Fix  a finite set of monomial generators $\{\chi^{u_1},\dots,\chi^{u_n}\}$ on $\mathscr{U}_{P}=\text{Spec}_{\ \!}K[M]^{P}$. Then $\{\chi^{u_1},\dots,\chi^{u_n},\frac{1}{g}\}$ is a set of topological generators on $\mathfrak{U}_{P}\langle\tfrac{1}{g}\rangle$. Define $U := (\mathscr{U}_{P})_{K}$, the affine algebraic generic fiber of the algebraic $R$-model $\mathscr{U}_{P}$.

Let $X = U_1 \cup \cdots \cup U_r$ be an open cover for which the inverse system $\mathcal{S}$ satisfies $(\dagger)$. Cover $U_i \cap U$ by open subsets $U_{ij}$ that are distinguished affine opens of both $U$ and $U_i$. In particular, $U_{ij} = (U_i)_{f_{ij}}$ for some $f_{ij} \in K[U_i]$.

Since $\ad{\mathfrak{U}_{P}\langle\tfrac{1}{g}\rangle}$ is a quasicompact subset of $\ad{U}$, there is a rational function $f$ on $U$ such that $|g^{-1}(x)|\le1$ if and only if $|f(x)|\le 1$ for all $x\in\an{\mathfrak{U}_{P}\langle\tfrac{1}{g}\rangle}$. Hence, we may assume $g$ is a rational function.  Write $g|_{U_{ij}} = g_{ij}/g'_{ij}$, with $g_{ij}$ and $g'_{ij}$ in $K[U_i]$.  

Choose functions $h_{ij}$ and $h'_{ij}$ in $K[U_i]$ such that $\an{\mathfrak{U}_{P}}$ is the subset of $\an{U_i}$ where $|h_{ij} / h'_{ij}| \leq 1$.

By $(\dagger)$, we can choose a toric embedding $\jmath\colon X \hookrightarrow Y_\Sigma$ in $\mathcal{S}$ such that $U_i$ is the preimage of an invariant affine open $U_{\sigma_i} \subset Y_\Sigma$ and each of the functions $f_{ij}$, $g_{ij}$, $g'_{ij}, h_{ij}, h'_{ij}$ is the pullback of a character that is regular on $U_{\sigma_i}$. Since $f_{ij}$ is a character that is regular on $U_i$, the distinguished affine open $U_{ij}$ is also the preimage of an invariant affine open $U_{\sigma_{ij}} \subset Y_\Sigma$.  

By construction, the intersection of $\an{\mathfrak{U}_P\big \langle \frac{1}{g} \big \rangle}$ with $\an{U_{ij}}$ is the preimage of an extended $(\Gamma,\Sigma)$-admissible polyhedron $\overline{Q}_{ij}$. Since $\mathcal{S}$ is an inverse system, we can then choose $\iota' \colon X \hookrightarrow Y_{\Sigma'}$ in $\mathcal{S}$ that dominates both $\iota$ and $\jmath$. The preimage of $\overline{Q}_{ij}$ is covered by finitely many $(\Gamma, \Sigma')$-admissible polyhedra obtained by intersecting with closures of translates of cones in $\Sigma'$, as is the preimage of each face of $\Delta$.

Applying Lemma~\ref{lem:refinements}, we obtain a locally finite $(\Gamma,\Sigma')$-admissible polyhedral complex $\Delta'$ that covers $\text{Trop}(X,\iota)$ and such that the preimage of each $\overline Q_{ij}$ and the preimage of each face of $\Delta$ is a union of faces of $\Delta'$. We then have a morphism of admissible formal $R$-models $\mathfrak{X}_{\Delta'}\longrightarrow\mathfrak{X}_{\Delta}$.



	Let $\Upsilon \subset \Delta'$ be the subcomplex consisting of the preimages of the $\overline Q_{ij}$. By construction, we have an isomorphism of ringed spaces
	\begin{equation}\label{equation: iso of adic fibers}
	\big(\ad{\mathfrak{X}_{\Upsilon}},\mathscr{O}^{\circ}_{\!\!\ad{\mathfrak{X}_{\Upsilon}}}\big)
	\ \xrightarrow{\ \sim\ }\ 
	\big(\ \!\ad{\mathfrak{U}_{P}\langle\tfrac{1}{g}\rangle},\mathscr{O}^{\circ}_{\!\ad{\mathfrak{U}_{P}\langle\mbox{{\smaller\smaller\smaller\smaller $\tfrac{1}{g}$}}\rangle}}\big).
	\end{equation}
To see that \eqref{equation: iso of adic fibers} extends to a morphism of admissible formal $R$-models $\mathfrak{X}_{\Upsilon}\longrightarrow\mathfrak{U}_{P}\langle\tfrac{1}{g}\rangle$, it suffices to show that each of the topological generators $\chi^{u_1},\dots,\chi^{u_n},\frac{1}{g}$ on $\mathfrak{U}_{P}\langle\tfrac{1}{g}\rangle$ is in $\mathcal{O}_{\mathfrak{X}_{\Upsilon}}$. Now $\chi^{u_i}$ is in $\mathcal{O}_{\mathfrak{X}_{\Upsilon}}$ by construction, and $g$ is invertible in $\mathcal{O}_{\mathfrak{X}_{\Upsilon}}$.  By Lemma~\ref{lemma: presence of inverse}, this implies that $g^{-1}$ is also in $\mathcal{O}_{\mathfrak{X}_{\Upsilon}}$, as required.
\end{proof}
	
\begin{lemma}\label{lem:for proof of Thm 1.3}
	Let $\Sc$ be an inverse system of toric embeddings of $X$ that satisfies $(\dagger)$, and let $V \subset \an{X}$ be a compact analytic domain with formal model $\mathfrak{V}$. For any point $x\in \ad{V}$, there exists an embedding $\iota\colon X \hookrightarrow Y_\Sigma$ in $\mathcal{S}$, an admissible polyhedral cover $\Delta_{x}$ of $\Trop{X}{\iota}$, a polyhedron $P_{x}\in \Delta_{x}$, and a function $g_{x}$ on $\mathfrak{U}_{P_{x}}$ such that the distinguished open $\mathfrak{U}_{P_{x}}\langle\tfrac{1}{g_{x}}\rangle$ comes with a morphism $\mathfrak{U}_{P_{x}}\langle\tfrac{1}{g_{x}}\rangle\lra\mathfrak{V}$ whose adic Raynaud fiber is the inclusion of an analytic domain containing $x$.
\end{lemma}
\begin{proof}
	For each $x\in \mathfrak{V}^{\text{ad}}$, choose an affine open subscheme $\mathfrak{V}'=\text{Spf}_{\ \!}A'\subset\mathfrak{V}$ such that $x\in(\mathfrak{V}')^{\text{ad}}$. Fix a finite set of topological generators $\{h_{1},\dots,h_{\ell}\}$ of $A'$. Note that all generators $h_{i}$, $1\le i\le \ell$, lie in the power bounded stalk $\mathscr{O}^{\circ}_{\mathfrak{V}^{\text{ad}},x}$. By Theorem \ref{thm:adiclimit}, our hypothesis that $\Sc$ is an inverse system of toric embeddings of $X$ satisfying $(\dagger)$ implies that we have a neighborhood basis of $x$ formed by adic fibers of distinguished affine opens $\mathfrak{U}_{P}\langle\tfrac{1}{g}\rangle$, where $\overline P$ is a face of an extended $(\Gamma,\Sigma)$-admissible complex giving rise to a Gubler model of $\an{X}$.  Thus, because we have only finitely many $h_{i}$'s, we can find a single admissible formal $R$-scheme $\mathfrak{U}_{P_{x}}\langle\tfrac{1}{g_{x}}\rangle$ inside some formal Gubler model $\mathfrak{X}_{\Delta_{x}}$ of $X$ associated to an admissible pair $(\iota,\Delta_{x})$, for $\iota$ in $\mathcal{S}$, such that $x\in\mathfrak{U}_{P_{x}}\langle\tfrac{1}{g_{x}}\rangle^{\text{ad}}$, and such that each $h_{i}$ is a function on $\mathfrak{U}_{P_{x}}\langle\tfrac{1}{g_{x}}\rangle$, for $1\le i\le \ell$. By construction, this comes with a morphism $\mathfrak{U}_{P_{x}}\langle\tfrac{1}{g_{x}}\rangle\lra\mathfrak{V}$, such that $x$ is in the image of the adic Raynaud fiber of this morphism.
\end{proof}
	
\begin{proof}[{\it Proof of Theorem \ref{theorem: cofinality of Gubler models}}]
	For each $x\in V^{\text{ad}}$, construct $\mathfrak{U}_{P_{x}}\langle\tfrac{1}{g_{x}}\rangle\lra\mathfrak{V}$ as in Lemma \ref{lem:for proof of Thm 1.3}. The fact that $\ad{V}$ is quasicompact implies that we can pass to a finite cover of $\ad{V}$ by strictly affinoid domains $\ad{\mathfrak{U}_{P_{x}}\langle\tfrac{1}{g_{x}}\rangle}\subset\ad{V}$. By Proposition \ref{proposition: Theorem 1.3 for distinguished opens}, each $\mathfrak{U}_{P_{x}}\langle\tfrac{1}{g_{x}}\rangle$ has a second associated Gubler model $\mathfrak{X}_{\Upsilon_{x}}$ of $X$ with dominating morphism $\mathfrak{X}_{\Upsilon_{\!x}}\longrightarrow\mathfrak{X}_{\Delta_{x}}$, and containing a subcomplex $\Upsilon'_{\!x}\subset\Upsilon_{x}$ with restriction $\mathfrak{X}_{\Upsilon'_{\!x}}\longrightarrow\mathfrak{U}_{P_{x}}\langle\tfrac{1}{g_{x}}\rangle$ that induces an isomorphism
	$$
	\ad{\mathfrak{X}_{\Upsilon'_{\!x}}}\xrightarrow{\ \sim\ }\ad{\mathfrak{U}_{P_{x}}\langle\tfrac{1}{g_{x}}\rangle}.
	$$
Because the system of Gubler models is cofiltered, there exists a single formal Gubler model $\mathfrak{X}_{\Upsilon}$ dominating all $\mathfrak{X}_{\Upsilon'_{\!x}}$ simultaneously. Let $\Upsilon'\subset\Upsilon$ denote the subcomplex consisting of all polyhedra mapping into $\Upsilon'_{\!x}$ for some $x$. Then by construction, we have a morphism $\mathfrak{X}_{\Upsilon'}\longrightarrow\mathfrak{V}$ that restricts to an isomorphism of adic fibers $\ad{\mathfrak{X}_{\Upsilon'}}\xrightarrow{\ \sim\ }V^{\text{ad}}$.
\end{proof}

\end{subsection}


\section{The limit theorem for structure sheaves on Berkovich spaces}\label{section: The limit theorem for structure sheaves on Berkovich spaces}


Recall that a subset $V \subset \an{X}$ is an \emph{analytic domain} if each point $x \in V$ has a neighborhood of the form $V_1 \cup \cdots \cup V_n$, where each $V_i \subset V$ is strictly affinoid, and $x \in V_1 \cap \cdots \cap V_n$.  A cover of an analytic domain $V$ by analytic subdomains $V_i$ is \emph{admissible} if every point $x \in V$ has a neighborhood which is a finite union $V_1 \cup \cdots \cup V_n$, where each $V_i$ is an analytic domain in the cover, and $x \in V_1 \cap \cdots \cap V_n$.   We write $\an{X}_G$ for $\an{X}$ equipped with this $G$-topology.  Note that open subsets are analytic domains, so the $G$-topology refines the ordinary topology.  See \cite{Berkovich90, Berkovich93} for the $G$-topology on Berkovich spaces, as well as \cite[\S9.1]{BGR84} for further details on Grothendieck topologies in general, including the technical notion of \emph{slightly finer} Grothendieck topologies, which appears in the discussion below.

Note that sheaves on $\an{X}_G$ are determined by their values on strictly affinoid domains and restriction maps for inclusions of strictly affinoid subdomains.  The value of the structure sheaf $\mathscr{O}(X_G)$ on a strictly affinoid domain $\mathscr{M}(A)$ is simply the affinoid algebra $A$, and the restriction maps to affinoid subdomains are the usual ones.

\subsection{The tropical $G$-topology}  We now describe an analogous $G$-topology on tropicalizations.  Let $\iota\colon X \hookrightarrow Y_\Sigma$ be a closed embedding into a toric variety over $K$.  A \emph{polyhedral domain} in $\Trop{X}{\iota}$ is the intersection with a $(\Gamma, \Sigma)$-admissible polyhedron.  These are the analogues of strictly affinoid domains in $\an{X}$.  Then a \emph{tropical domain} $W$ is a subset of $\Trop{X}{\iota}$ in which every point $x \in W$ has a neighborhood which is a finite union $W_1 \cup \cdots \cup W_n$, where each $W_i \subset W$ is a polyhedral domain and $x \in W_1 \cap \cdots \cap W_n$.  A cover of a tropical domain $W$ by tropical subdomains $W_i$ is \emph{admissible} if every point $x \in W$ has a neighborhood which is a finite union $W_1 \cup \cdots \cup W_n$, where each $W_i$ is a tropical domain in the cover, and $x \in W_1 \cap \cdots \cap W_n$.  We write $\Trop{X}{\iota}_G$ for the tropicalization equipped with this $G$-topology.  Just as a sheaf on $\an{X}_G$ is determined by its values on strictly affinoid domains and the restriction maps between them, a sheaf on $\Trop{X}{\iota}_G$ is determined by its values on polyhedral domains and the restriction maps between them.

Note that the preimage of a polyhedral (resp. tropical) domain in $\Trop{X}{\iota}$ is a strictly affinoid (resp. analytic) domain in $\an{X}$.  Hence the pullback of an admissible cover of $\Trop{X}{\iota}$ is an admissible cover of $\an{X}$, so the tropicalization map
\[
\an{X}_G \longrightarrow \Trop{X}{\iota}_G
\]
is continuous not only in the ordinary topologies, but also with respect to the $G$-topologies on both sides.  The projections
\[
\Trop{X}{\iota'}_G \longrightarrow \Trop{X}{\iota}_G
\]
induced by morphisms of toric embeddings are likewise continuous with respect to the $G$-topologies.

The set-theoretic identification $\an{X} = \varprojlim  \Trop{X}{\iota}$, together with the $G$-topologies on the tropicalizations, induces a $G$-topology on $\an{X}$, which we call the \emph{tropical topology}.  It is the coarsest $G$-topology on $\an{X}$ with respect to which each of the projections $\an{X}\rightarrow\Trop{X}{\iota}_G$ is continuous.  

\begin{proposition}\label{proposition: G is slightly finer}
The $G$-topology on $\an{X}$ is slightly finer than the tropical topology on $\an{X}$.
\end{proposition}

\noindent The following lemma is a key step in the proof of the proposition.

\begin{lemma}\label{lem:cover by tropical domains}
	Every analytic domain $V\subset \an{X}$ admits an admissible cover by analytic domains that are admissible in the tropical topology.
\end{lemma}
\begin{proof}
	Choose an admissible cover $V=\bigcup\mathscr{M}(K\otimes_{R}A_{i})$ by strictly affinoid domains, where each $A_{i}$ is an admissible $R$-algebra. It suffices to prove that each $\mathscr{M}(K\otimes_{R}A_{i})$ has an admissible tropical cover. We now use the same strategy that we used in the proof of Theorem  \ref{theorem: cofinality of Gubler models}.
	
	Let $\mathfrak{V} = \mathrm{Spf} A_i$. For each $x\in \mathscr{M}(K\otimes_{R}A_{i})^{\text{ad}}$, construct $\mathfrak{U}_{P_{x}}\langle\tfrac{1}{g_{x}}\rangle\lra\mathfrak{V}$ as in the statement of Lemma \ref{lem:for proof of Thm 1.3}, using the inverse system $\mathcal{S}$ of all closed embeddings of $X$ into toric varieties. Recall that $\mathfrak{U}_{P_{x}}$ is a formal polyhedral domain in a Gubler model $\mathfrak{X}_{\Delta_{x}}$ of $X$. By Proposition \ref{proposition: Theorem 1.3 for distinguished opens}, we can find a formal Gubler model $\mathfrak{X}_{\Delta'_{x}}$ of $X$ that is adapted to $\an{\mathfrak{U}_{P_{x}} \big \langle \frac{1}{g_{x}} \big \rangle}$ and dominates $\mathfrak{X}_{\Delta_{x}}$. Since $\mathscr{M}(K\otimes_{R}A_{i})^{\text{ad}}$ is quasicompact, we can choose a finite collection of points $x\in\mathscr{M}(K\otimes_{R}A_{i})$ such that the adic open subspaces $\ad{\mathfrak{U}_{P_{x}} \big \langle \frac{1}{g_{x}} \big \rangle}$ cover $\mathscr{M}(K\otimes_{R}A_{i})^{\text{ad}}$, and thus the analytic domains $\an{\mathfrak{U}_{P_{x}} \big \langle \frac{1}{g_{x}} \big \rangle}$ cover $\mathscr{M}(K\otimes_{R}A_{i})^{\text{an}}$. Choose a single Gubler model $\mathfrak{X}_{\Upsilon}$ dominating the models $\mathfrak{X}_{\Delta'_{x}}$ for this finite collection of points $x\in\mathscr{M}(K\otimes_{R}A_{i})^{\text{ad}}$. Let $\Upsilon'\subset\Upsilon$ be the subcomplex consisting of all polyhedra $Q\in\Upsilon$ such that $\mathfrak{X}^{\text{an}}_{Q}$ lies in one of the analytic domains $\an{\mathfrak{U}_{P_{x}} \big \langle \frac{1}{g_{x}} \big \rangle}$. Then the collection $\big\{\mathfrak{X}_{Q}^{\text{an}}:Q\in\Upsilon'\big\}$ is an admissible tropical cover of $\mathscr{M}(K\otimes_{R}A_{i})$.
\end{proof}

\begin{proof}[Proof of Proposition~\ref{proposition: G is slightly finer}]
	We need to check criteria (i) through (iii) of \cite[Definition~9.1.2.1]{BGR84}.
	To see that (i) holds, observe that every admissible open in the tropical topology is a strictly affinoid and hence an analytic domain. Since every tropical admissible cover is $G$-admissible, this implies that the $G$-topology is finer than the tropical topology on $\an{X}$. Criterion (ii) follows immediately from Lemma \ref{lem:cover by tropical domains}. To verify (iii), we need to show that any $G$-covering of a tropical domain has a tropical refinement. Let $V$ be a tropical domain, with $\{V_i \}_{i \in I}$ a $G$-admissible cover by analytic domains.  Then each point $x \in V$ has a neighborhood in $V$ which is a finite union of analytic domains $V_{i1} \cup \cdots \cup V_{ir}$, where each $V_{ij}$ contains $x$. By Lemma \ref{lem:cover by tropical domains}, each of the analytic domains $V_{ij}$ admits an admissible tropical cover. This shows that every point in $V$ has a neighborhood which is a finite union of tropical domains, each contained in one of the analytic domains $V_i$.  All of these tropical domains together form an admissible tropical cover of $V$ that refines the $G$-admissible cover $\{V_i\}$, as required.
\end{proof}

Sheaves in a Grothendieck topology extend uniquely to sheaves in slightly finer Grothendieck topologies \cite[Proposition~9.2.3.1]{BGR84}, so sheaves on $\an{X}_G$ are determined by their values on tropical domains and the restriction maps between them.

There is a natural structure sheaf $\mathscr{O}^\mathrm{trop}$ in the tropical topology on $\an{X}$, which records the structure sheaves of all formal Gubler models associated to admissible subdivisions of $\Trop{X}{\iota}$, as follows: Let $\overline{P}$ be the closure of a $\Gamma$-admissible polyhedron with recession cone in $\Sigma$, and let $V \subset \an{X}$ be the preimage of $\overline{P}$ under the tropicalization map $X^\text{an}\to\Trop{X}{\iota}$. By \S\ref{section: Gubler models of closed subvarieties}, we have $V=\mathfrak{X}^\text{an}_{P}$.  Then we define
	\[
	\mathscr{O}^\mathrm{trop}(V)
	\ :=\ 
	\text{coordinate ring of }\mathfrak{X}_{P}.
	\]
Note that the global analytic functions on $V$ are $K[V] = K\otimes_{R}\mathscr{O}^\mathrm{trop}(U)$.

\begin{corollary}\label{corollary: recovering analytic sheaf from tropical sheaf}
The analytic structure sheaf $\mathscr{O}_{\an{X}_G}$ is the unique sheaf in the $G$-topology that extends the sheaf $K \otimes_R \mathscr{O}^\mathrm{trop}$ in the tropical topology.
\hfill
$\square$
\end{corollary}

\begin{remark}
	Fix a pair of $K$-varieties $X$ and $Y$. Because morphisms $\an{X}\longrightarrow\an{Y}$ are exactly morphisms of ringed spaces that are locally dual to bounded morphisms of strictly affinoid $K$-algebras in the $G$-topologies on $\an{X}$ and $\an{Y}$,  Proposition \ref{proposition: G is slightly finer} and its Corollary \ref{corollary: recovering analytic sheaf from tropical sheaf} imply that $\text{Hom}_{K}(\an{X},\an{Y})$ is determined locally in the tropical topologies on $X$ and $Y$. In this sense, we can recover the entire category of analytifications using the tropical topologies on analytifications.
\end{remark}

\begin{proof}[{\it Proof of Theorem \ref{thm:berklimit}.}]
	Let $X^{\text{an}}_{\text{trop}}$ denote $\an{X}$ equipped with its tropical topology. By \cite[Proposition~9.2.3.1]{BGR84}, restriction to the tropical topology induces an equivalence of topoi
	\begin{equation}\label{equation: equivalence of topoi}
	\text{Sh}(X^{\text{an}}_{G})\xrightarrow{\ \sim\ }\text{Sh}(X^{\text{an}}_{\text{trop}}).
	\end{equation}
The sheaf $\mathscr{O}_{\an{X}_{G}}$ is a ring object in the topos $\text{Sh}(\an{X}_{G})$, and the sheaf $\mathscr{O}^{\text{trop}}$ is a ring object in the topos $\text{Sh}(X^{\text{an}}_{\text{trop}})$. By definition, each gives its respective topos the structure of a ringed topos. By Corollary \ref{corollary: recovering analytic sheaf from tropical sheaf}, the equivalence \eqref{equation: equivalence of topoi} takes
	$$
	\mathscr{O}_{\an{X}_{G}}
	\longmapsto
	K \otimes_R \mathscr{O}^{\text{trop}},
	$$
and is thus an equivalence of ringed topoi.
	
	The fact that the functor \eqref{equation: equivalence of topoi} is restriction to a slightly coarser topology implies that at each point $x\in\an{X}$, we have an isomorphism of stalks
	$$
	\mathscr{O}_{\an{X}_{G},x}\ \cong\ \mathscr{O}^{\text{trop}}_{x}.
	$$
Because $\mathscr{O}_{\an{X}_{G}}$ is a sheaf of local rings, this makes \eqref{equation: equivalence of topoi} an equivalence of locally ringed topoi.
\end{proof}


\bibliographystyle{amsalpha}
\bibliography{math}

\providecommand{\bysame}{\leavevmode\hbox to3em{\hrulefill}\thinspace}
\providecommand{\MR}{\relax\ifhmode\unskip\space\fi MR }
\providecommand{\MRhref}[2]{%
  \href{http://www.ams.org/mathscinet-getitem?mr=#1}{#2}
}
\providecommand{\href}[2]{#2}
\begin{thebibliography}{GRW17}

\bibitem[AB15]{AminiBaker15}
O.~Amini and M.~Baker, \emph{Linear series on metrized complexes of algebraic
  curves}, Math. Ann. \textbf{362} (2015), no.~1-2, 55--106.

\bibitem[Abb10]{Abbes10}
A.~Abbes, \emph{\'{E}l\'ements de g\'eom\'etrie rigide. {V}olume {I}}, Progress
  in Mathematics, vol. 286, Birkh\"auser/Springer Basel AG, Basel, 2010,
  Construction et {\'e}tude g{\'e}om{\'e}trique des espaces rigides.
  [Construction and geometric study of rigid spaces], With a preface by Michel
  Raynaud.

\bibitem[Ber90]{Berkovich90}
V.~Berkovich, \emph{Spectral theory and analytic geometry over
  non-{A}rchimedean fields}, Mathematical Surveys and Monographs, vol.~33,
  American Mathematical Society, Providence, RI, 1990.

\bibitem[Ber93]{Berkovich93}
\bysame, \emph{\'{E}tale cohomology for non-{A}rchimedean analytic spaces},
  Inst. Hautes \'Etudes Sci. Publ. Math. (1993), no.~78, 5--161 (1994).

\bibitem[BGR84]{BGR84}
S.~Bosch, U.~G{\"u}ntzer, and R.~Remmert, \emph{Non-{A}rchimedean analysis},
  Grundlehren der Mathematischen Wissenschaften, vol. 261, Springer-Verlag,
  Berlin, 1984.

\bibitem[Bos14]{Bosch14}
S.~Bosch, \emph{Lectures on formal and rigid geometry}, Lecture Notes in
  Mathematics, vol. 2105, Springer, Cham, 2014.

\bibitem[CF23a]{ColesFriedenberg23b}
D.~Coles and N.~Friedenberg, \emph{A construction of algebraizable formal
  models}, preprint arXiv:2303.13646, 2023.

\bibitem[CF23b]{ColesFriedenberg23}
\bysame, \emph{Locally finite completions of polyhedral complexes}, preprint
  arXiv:2303.12334, 2023.

\bibitem[FGP14]{limits}
T.~Foster, P.~Gross, and S.~Payne, \emph{Limits of tropicalizations}, Israel J.
  Math. \textbf{201} (2014), no.~2, 835--846.

\bibitem[Fos16]{Foster16}
T.~Foster, \emph{Introduction to adic tropicalization}, Nonarchimedean and
  tropical geometry (M.~Baker and S.~Payne, eds.), Simons Symposia, Springer,
  2016.

\bibitem[FR16a]{FosterRanganathan16b}
T.~Foster and D.~Ranganathan, \emph{Degenerations of toric varieties over
  valuation rings}, Bull. Lond. Math. Soc. \textbf{48} (2016), no.~5, 835--847.

\bibitem[FR16b]{FosterRanganathan16}
\bysame, \emph{Hahn analytification and connectivity of higher rank tropical
  varieties}, Manuscripta Math. \textbf{151} (2016), no.~3-4, 353--374.

\bibitem[GRW17]{GublerRabinoffWerner17}
W.~Gubler, J.~Rabinoff, and A.~Werner, \emph{Tropical skeletons}, Ann. Inst.
  Fourier (Grenoble) \textbf{67} (2017), no.~5, 1905--1961.

\bibitem[GS15]{GublerSoto15}
W.~Gubler and A.~Soto, \emph{Classification of normal toric varieties over a
  valuation ring of rank one}, Doc. Math. \textbf{20} (2015), 171--198.

\bibitem[Gub13]{Gubler13}
W.~Gubler, \emph{A guide to tropicalizations}, Algebraic and combinatorial
  aspects of tropical geometry, Contemp. Math., vol. 589, Amer. Math. Soc.,
  Providence, RI, 2013, pp.~125--189.

\bibitem[Hub94]{Huber94}
R.~Huber, \emph{A generalization of formal schemes and rigid analytic
  varieties}, Math. Z. \textbf{217} (1994), no.~4, 513--551.

\bibitem[Hub96]{Huber96}
\bysame, \emph{\'{E}tale cohomology of rigid analytic varieties and adic
  spaces}, Aspects of Mathematics, E30, Friedr. Vieweg \& Sohn, Braunschweig,
  1996.

\bibitem[Par12]{Parker12}
B.~Parker, \emph{Exploded manifolds}, Adv. Math. \textbf{229} (2012), no.~6,
  3256--3319.

\bibitem[Pay09]{analytification}
S.~Payne, \emph{Analytification is the limit of all tropicalizations}, Math.
  Res. Lett. \textbf{16} (2009), no.~3, 543--556.

\bibitem[RG71]{RaynaudGruson71}
M.~Raynaud and L.~Gruson, \emph{Crit\`eres de platitude et de projectivit\'e.
  {T}echniques de ``platification'' d'un module}, Invent. Math. \textbf{13}
  (1971), 1--89.

\bibitem[Sch12]{Scholze12}
P.~Scholze, \emph{Perfectoid spaces}, Publ. Math. Inst. Hautes \'Etudes Sci.
  \textbf{116} (2012), 245--313.

\bibitem[W{\l}o93]{Wlodarczyk93}
J.~W{\l}odarczyk, \emph{Embeddings in toric varieties and prevarieties}, J.
  Algebraic Geom. \textbf{2} (1993), no.~4, 705--726.

\end{thebibliography}

\end{document}